\documentclass[12pt]{amsart}

\usepackage{amsthm}
\usepackage{calc}
\usepackage{amsmath}
\usepackage{amssymb}
\usepackage{words}
\usepackage[inline]{enumitem}
\usepackage[all]{xy}
\usepackage[colorlinks=true]{hyperref}
\usepackage{widebar}

\DeclareMathOperator{\GS}{GS}

\def\marg{\footnote}

\newtheorem{theorem}{Theorem}
\newtheorem{lemma}[theorem]{Lemma}
\newtheorem{proposition}[theorem]{Proposition}
\newtheorem{corollary}{Corollary}
\newtheorem{sublemma}[corollary]{Lemma}
\numberwithin{corollary}{theorem}
\numberwithin{theorem}{section}
\numberwithin{equation}{section}
\theoremstyle{definition}
\newtheorem{definition}[theorem]{Definition}
\newtheorem{convention}[theorem]{Convention}
\theoremstyle{remark}
\newtheorem{remark}{Remark}
\numberwithin{remark}{section}

\newcommand{\fpr}{\mathbin{\times}}

\DeclareMathOperator{\Log}{\mathbf{Log}}
\def\f#1{\mathfrak{#1}}

\begin{document}

\title{Moduli of morphisms of logarithmic schemes}
\author{Jonathan Wise}
\address{University of Colorado, Boulder\\
Boulder, Colorado 80309-0395\\ USA}
\thanks{Wise is supported by an NSA Young Investigator's Grant, Award \#H98230-14-1-0107.}
\subjclass[2010]{14H10, 
 14D23, 
 14A20 
}
\date{\today}
\begin{abstract}
We show that there is a logarithmic algebraic space parameterizing logarithmic morphisms between fixed logarithmic schemes when those logarithmic schemes satisfy natural hypotheses.  As a corollary, we obtain the representability of the stack of stable logarithmic maps from logarithmic curves to a fixed target without restriction on the logarithmic structure of the target.

An intermediate step requires a left adjoint to pullback of \'etale sheaves, whose construction appears to be new in the generality considered here, and which may be of independent interest.
\end{abstract}

\maketitle

\setcounter{tocdepth}{1}
\tableofcontents

\section{Introduction}

Let $X$ and $Y$ be logarithmic algebraic spaces over a logarithmic scheme $S$.  Consider the functor $\Hom_{\LogSch / S}(X,Y)$%
\footnote{See Section~\ref{sec:conventions} for our conventions on notation.}
whose value on a logarithmic $S$-scheme $S'$ the set of logarithmic morphisms $X' \rightarrow Y'$, where $X' = X \fpr_S S'$ and $Y' = Y \fpr_S S'$.  Under reasonable hypotheses on these data, we show that $\Hom_{\LogSch/S}(X,Y)$ is representable by a logarithmic algebraic space over $S$.

Our strategy is to work relative to the space $\Hom_{\LogSch/S}(\u{X},\u{Y})$ parameterizing morphisms of the underlying algebraic spaces $\u{X}$ and $\u{Y}$ of $X$ and $Y$, respectively.  More precisely, $\Hom_{\LogSch/S}(\u {X},\u {Y})(S')$ is the set of morphisms of schemes $\u {X} \fpr_{\u {S}} \u {S}' \rightarrow \u {Y} \fpr_{\u {S}} \u {S}'$.  In order to guarantee that a morphism
\begin{equation*}
\Hom_{\LogSch/S}(X,Y) \rightarrow \Hom_{\LogSch/S}(\u {X},\u {Y})
\end{equation*}
exists, we need to assume that the morphism of logarithmic spaces $\pi : X \rightarrow S$ is integral, meaning $\pi^\ast M_S \rightarrow M_X$ is an integral morphism of sheaves of monoids.

\begin{theorem} \label{thm:main}
Let $\pi : X \rightarrow S$ be a proper, flat, finite presentation, geometrically reduced, integral morphism of fine logarithmic algebraic spaces.  Let $Y$ be a logarithmic stack%
\footnote{It is not necessary for $Y$ to be algebraic.}
over $S$.  Then the morphism
\begin{equation*}
\Hom_{\LogSch/S}(X,Y) \rightarrow \Hom_{\LogSch/S}(\u X, \u Y)
\end{equation*}
is representable by logarithmic algebraic spaces.
\end{theorem}

Combining the theorem with already known criteria for the algebraicity of $\Hom_{\LogSch/S}(\u X, \u Y)$ such as \cite[Theorem~1.2]{HR-Hom}, we obtain

\begin{corollary}
In addition to the assumptions of Theorem~\ref{thm:main}, assume as well that $Y$ is an algebraic stack over $S$ that is locally of finite presentation with quasi-compact and quasi-separated diagonal and affine stabilizers and that $X \rightarrow S$ is of finite presentation.  Then $\Hom_{\LogSch/S}(X,Y)$ is representable by a logarithmic algebraic stack.
\end{corollary}

One application is to the construction of the stack of pre-stable logarithmic maps.  Let $\f M$ denote the logarithmic stack of logarithmic curves.  The algebraicity of $\f M$ may be verified in a variety of ways, e.g., \cite[Proposition~A.3]{GS}.  For a logarithmic algebraic stack $Y$ over $S$, we write $\fM(Y/S)$ for the logarithmic stack whose $T$-points are logarithmically commutative diagrams
\begin{equation*} \xymatrix{
C \ar[r] \ar[d] & Y \ar[d] \\
T \ar[r] & S
} \end{equation*}
in which $C$ is a logarithmic curve over $T$.

Taking $X$ to be the universal curve over $\f M$ in the previous corollary yields

\begin{corollary}
Suppose that $Y \rightarrow S$ is a morphism of logarithmic algebraic stacks with quasi-finite and separated relative diagonal.  Then $\fM(Y/S)$ is representable by a logarithmic algebraic stack.
\end{corollary}

This improves on several previous results:
\begin{enumerate}
\item \cite{Chen} required $Y$ to have a rank~$1$ Deligne--Faltings logarithmic structure, 
\item \cite{AC} required $Y$  to have a generalized Deligne--Faltings logarithmic structure, and 
\item \cite{GS} required $Y$ to have a Zariski logarithmic structure.
\end{enumerate}

The evaluation space for stable logarithmic maps can also be constructed using Theorem~\ref{thm:main}.  Recall that the standard logarithmic point $P$ is defined by restricting the divisorial logarithmic structure of $\mathbf{A}^1$ to the origin.  A family of standard logarithmic points in $Y$ parameterized by a logarithmic scheme $S$ is a morphism of logarithmic algebraic stacks $S \times P \rightarrow Y$.  Following \cite{ACGM}, we define $\wedge Y$ to be the fibered category of standard logarithmic points of $Y$.

\begin{corollary}[{\cite[Theorem~1.1.1]{ACGM}}]
If $Y$ is a logarithmic algebraic stack with quasi-finite and quasi-separated diagonal then $\wedge Y$ is representable by an algebraic stack with logarithmic structure.
\end{corollary}

\subsection{Outline of the proof}

Working relative to $\Hom_{\LogSch/S}(\u X, \u Y)$, the question of the algebraicity of $\Hom_{\LogSch/S}(X,Y)$ is reduced to showing that, given a logarithmic algebraic space $X$ and a logarithmic algebraic stack $Y$, both over $S$, as well as a commutative triangle of algebraic stacks,
\begin{equation*} \xymatrix{
\u {X} \ar[dr]_{\pi} \ar[rr]^{f} & & \ar[dl] \u {Y} \\
& \u {S}
}\end{equation*}
the lifts of $f$ to an $S$-morphism of logarithmic algebraic stacks making the triangle commute are representable by a logarithmic algebraic stack.

This problem reduces immediately to the verification that morphisms of logarithmic structures $f^\ast M_Y \rightarrow M_X$ compatible with the maps \emph{from} $\pi^\ast M_S$ are representable by a logarithmic algebraic space over $S$.  We may therefore eliminate $Y$ from our consideration by setting $M = f^\ast M_Y$ and restricting our attention to the functor $\Hom_{\LogSch/S}(M,M_X)$ that parameterizes morphisms of logarithmic structures $M \rightarrow M_X$.

Stated precisely, the $S'$-points of $\Hom_{\LogSch/S}(M,M_X)$ are the morphisms of logarithmic structures $M' \rightarrow M'_X$, where $M'$ and $M'_X$ are the logarithmic structures deduced by base change on $\u X' = \u X \fpr_{\u S} \u S'$, that fit into a commutative triangle:
\begin{equation*} \xymatrix{
& {\pi'}^\ast M_{S'} \ar[dl] \ar[dr] \\
M' \ar[rr] & & M'_X
} \end{equation*}

As is typical for logarithmic moduli problems, we now separate the question of the representability of $\Hom_{\LogSch/S}(M,M_X)$ by a logarithmic algebraic stack into a question about the representability of a larger stack $\Log(\Hom_{\LogSch/S}(M,M_X)) = \Hom_{\Sch / \Log(S)}(M,M_X)$ over \emph{schemes} (not logarithmic schemes), followed by the identification of an open substack of \emph{minimal} objects within $\Log(\Hom_{\LogSch/S}(M,M_X))$ that represents $\Hom_{\LogSch/S}(M,M_X)$.

When $\Log(\Hom_{\LogSch/S}(M,M_X))$ is viewed as a category, its objects are the same as the objects of $\Hom_{\LogSch/S}(M,M_X)$.  Its fiber over a scheme $T$ therefore consists of tuples $(M_T,f,\alpha)$ where 
\begin{enumerate}[label=(\roman{*})]
\item $M_T$ is a logarithmic structure on $T$,
\item $f : (T,M_T) \rightarrow S$ is a morphism of logarithmic schemes, and
\item $\alpha : f^\ast M \rightarrow f^\ast M_X$ is a morphism of logarithmic structures on $\u X \fpr_{\u S} T$ that is compatible with the maps from $M_T$.
\end{enumerate}
The distinction between the categories $\Log(\Hom_{\LogSch/S}(M,M_X))$ and $\Hom_{\LogSch/S}(M,M_X)$ is that morphisms in the former are required to be cartesian over the category of schemes, while in the latter they are only required to be cartesian over the category of logarithmic schemes.  That is, $(T,M_T,f,\alpha) \rightarrow (T',M_{T'},f',\alpha')$ in $\Hom_{\LogSch/S}(M,M_X)$ lies in $\Log(\Hom_{\LogSch/S}(M,M_X))$ only if the map $(T,M_T) \rightarrow (T',M_{T'})$ is strict.

We show in Section~\ref{sec:log-rep} that $\Log(\Hom_{\LogSch/S}(M,M_X))$ is representable by an algebraic space relative to $\Log(S)$.  As Olsson has proved that $\Log(S)$ is algebraic, the representability of $\Log(\Hom_{\LogSch/S}(M,M_X))$ by an algebraic \emph{stack} follows.

In Sections~\ref{sec:local-case} and~\ref{sec:global-case} we use Gillam's criterion (Section~\ref{sec:minimality}) to prove that the fibered category $\Hom_{\LogSch/S}(M,M_X)$ over logarithmic schemes is induced from a fibered category over schemes with a logarithmic structure.  Section~\ref{sec:local-case} treats the case where $X = S$ by adapting methods from homological algebra to commutative monoids.  In Section~\ref{sec:global-case}, we transform the local minimal object of Section~\ref{sec:local-case} to a global minimal object by means of a left adjoint to pullback for \'etale sheaves (constructed, under suitable hypotheses, in Section~\ref{sec:pi_0}), whose existence appears to be a new observation.

Gillam's criterion characterizes the fibered category over schemes inducing the fibered category $\Hom_{\LogSch/S}(M,M_X)$ over logarithmic schemes:  it is the substack of \emph{minimal objects} of $\Hom_{\LogSch/S}(M,M_X)$.  A slight augmentation of that criterion (described in Section~\ref{sec:minimality}) implies that the substack of minimal objects is open in $\Log(\Hom_{\LogSch/S}(M,M_X))$.  Combined with the algebraicity of $\Log(\Hom_{\LogSch/S}(M,M_X))$ proved in Section~\ref{sec:log-rep}, the verification of Gillam's criteria in Sections~\ref{sec:local-case} and~\ref{sec:global-case} implies that $\Hom_{\LogSch/S}(M,M_X)$ is representable by a logarithmic algebraic \emph{stack}.  A direct analysis of the stabilizers of logarithmic maps in Section~\ref{sec:automorphisms} then implies that $\Hom_{\LogSch/S}(M,M_X)$ is representable by an algebraic \emph{space}.

\subsection{Remarks on hypotheses}
\label{sec:hypotheses}

It is far from clear that all of the hypotheses of Theorem~\ref{thm:main} are essential.  We summarize how they are used in the proof:  Integrality of $X$ over $S$ is used to guarantee the existence of a morphism $\Hom_{\LogSch/S}(X,Y) \rightarrow \Hom_{\LogSch/S}(\u X,\u Y)$; it is also used in the construction of minimal objects.  Properness and finite presentation are used to guarantee the representability of $\Log(\Hom_{\LogSch/S}(M,M_X))$ by an algebraic stack.  Flatness, finite presentation, and geometrically reduced fibers are used to guarantee the existence of a left adjoint to pullback of \'etale sheaves, used in the construction of global minimal objects from local ones.

\subsection{Conventions}
\label{sec:conventions}

We generally follow the notation of \cite{Kato} concerning logarithmic structures, except that we write $\u X$ for the scheme (or fibered category) underlying a logarithmic scheme (or fibered category) $X$.  The logarithmic structures that appear in this paper will all be fine, although we will usually point this out in context.  If $M$ is a logarithmic structure on $X$, we write $\exp : M \rightarrow \mathcal{O}_X$ for the structural morphism and $\log : \mathcal{O}_X^\ast \rightarrow M$ for the reverse inclusion.

It is occasionally convenient to pass only part of the way from a chart for a logarithmic structure to its associated logarithmic structure.  We formalize this in the following definition:

\begin{definition} \label{def:quasi-log}
A \emph{quasi-logarithmic structure} on an scheme $X$ is an extension $N$ of an \'etale sheaf of integral%
\footnote{The integrality assumption is not necessary in the definition.  It is included to avoid qualifying every quasi-logarithmic structure that appears below with the adjective `integral'.}
monoids $\o N$ by $\mathcal{O}_X^\ast$ and a morphism $N \rightarrow \mathcal{O}_X$ compatible with the inclusions of $\mathcal{O}_X^\ast$.  We will say that a quasi-logarithmic structure is \emph{coherent} if its associated logarithmic structure is coherent.  If $f : X' \rightarrow X$ is a morphism of schemes, the \emph{pullback} $f^\ast N$ of $N$ to $X'$ is obtained by pushout via $f^{-1} \mathcal{O}_X^\ast \rightarrow \mathcal{O}_{X'}$ from the pulled back extension $f^{-1} N$:
\begin{equation*} \xymatrix{
0 \ar[r] & f^{-1} \mathcal{O}_X^\ast \ar[r] \ar[d] & f^{-1} N \ar[r] \ar[d] & f^\ast \o N \ar[r] \ar@{=}[d] & 0 \\
0 \ar[r] & \mathcal{O}_{X'}^\ast \ar[r] & f^\ast N \ar[r] & f^\ast \o N \ar[r] & 0
} \end{equation*}
\end{definition}

We write $\Hom(A,B)$ for the set of morphisms between two objects of the same type.  When $A$ and $B$ and the morphisms between them may reasonably be construed to vary with objects of a category $\mathscr{C}$, we write $\Hom_{\mathscr{C}}(A,B)$ for the functor or fibered category of morphisms between $A$ and $B$.  Occasionally, we also employ a subscript on $\Hom$ to indicate restrict to homomorphisms preserving some additional structure.  We rely on context to keep the two meanings of these decorations distinct.

\subsection{Acknowledgements}

The central construction of this paper was inspired by the construction of basic logarithmic maps in \cite{GS}.  I am very grateful to Qile Chen and Steffen Marcus for their comments on an early draft, and to Dan Abramovich for his many detailed comments and his encouragement that I present these results in their natural generality.  Samouil Molcho discovered an error in my original presentation of the main construction.  I am very grateful to him for that correction, as well as for the attached Appendix~\ref{app:calc}, which translates the results of this paper into explicit formulas.

This work was supported by an NSA Young Investigator's grant, award number H98230-14-1-0107.

\section{Algebraicity relative to the category of schemes}
\label{sec:log-rep}

We show that the morphism
\begin{equation} \label{eqn:7}
\Log(\Hom_{\LogSch/S}(M,M_X)) \rightarrow \Log(S)
\end{equation}
is representable by algebraic spaces.  Combined with the algebraicity of $\Log(S)$ \cite[Theorem~1.1]{Olsson-Log}, this implies $\Log(\Hom_{\LogSch/S}(M,M_X))$ is representable by an algebraic stack.

\begin{proposition} \label{prop:log-rep}
Let $S = (\u {S}, M_S)$ be a logarithmic scheme, $X$ a proper $S$-scheme with $\pi : X \rightarrow S$ denoting the projection.  Assume given logarithmic structures $M$ and $M_X$ on $X$ with morphisms of logarithmic structures $\pi^\ast M_S \rightarrow M$ and $\pi^\ast M_S \rightarrow M_X$.  Assume as well that $M_S$, $M_X$, and $M$ are all coherent.  Then the morphism~\eqref{eqn:7} is representable by algebraic spaces.
\end{proposition}

We divide this into a local problem in which $X = S$ and then pass to the general case.

\begin{theorem}[{\cite[Proposition~2.9]{GS}}] \label{thm:log-hom}
Suppose that $P$ and $Q$ are coherent logarithmic structures on a scheme $X$.  Then $\Hom_{\Sch / X}(P,Q)$ is representable by an algebraic space over $X$.
\end{theorem}

For isomorphisms, this is \cite[Corollary~3.4]{Olsson-Log}.

\begin{proof}
The question of the algebraicity of $\Hom_{\Sch / X}(P,Q)$ may be separated into one about the algebraicity of $\Hom_{\Sch / X}(\o P, \o Q)$ and another about the relative algebraicity of the map
\begin{equation} \label{eqn:12}
\Hom_{\Sch / X}(P,Q) \rightarrow \Hom_{\Sch / X}(\o P, \o Q) .
\end{equation}

\begin{sublemma} \label{lem:log-char-hom}
The functor $\Hom_{\Sch/X}(\o P, \o Q)$ is representable by an \'etale algebraic space over $X$.
\end{sublemma}
\begin{proof}
Because $\o P$ and $\o Q$ are constructible, the natural map
\begin{equation*}
f^\ast \Hom_{\et(X)}(\o P, \o Q) \rightarrow \Hom_{\et(X')}(f^\ast \o P, f^\ast \o Q)
\end{equation*}
is an isomorphism for any morphism $f : X' \rightarrow X$.  Therefore we may represent $\Hom_{\Sch/X}(\o P, \o Q)$ with the espace \'etal\'e of $\Hom_{\et(X)}(\o P, \o Q)$.
\end{proof}

The relative algebraicity of~\eqref{eqn:12} is equivalent to the following lemma:

\begin{sublemma}[{\cite[Lemma~2.12]{GS}}] \label{lem:log-hom-fixed-char}
Let $Q$ be a logarithmic structure on a scheme $X$ and let $P$ be a coherent quasi-logarithmic structure on $X$.  Fix a morphism $\o u : \o P \rightarrow \o Q$.  The lifts of $\o u$ to a morphism of quasi-logarithmic structures $u : P \rightarrow Q$ are parameterized by a relatively affine scheme over $X$.
\end{sublemma}

Our proof of this lemma only differs from that of loc.\ cit.\ superficially, but is nevertheless included for the sake of completeness.  It is also possible to deduce Lemma~\ref{lem:log-hom-fixed-char} from Lemma~\ref{lem:log-char-hom} and \cite[Corollary~3.4]{Olsson-Log}.

\begin{proof}
This is a local question in $X$, so we may freely pass to an \'etale cover.  Furthermore, replacing $P$ with a quasi-logarithmic structure $P_0$ that has the same associated logarithmic structure does not change the morphisms to $Q$, by the universal property of the associated logarithmic structure.  Since the logarithmic structure associated to $P$ admits a chart \'etale locally, we can therefore select $P_0$ to be a quasi-logarithmic structure whose sheaf of characteristic monoids $\o P_0$ is constant.  Replacing $P$ with $P_0$, we can assume that the characteristic monoid of $P$ is constant.

We wish to construct the space of completions of the diagram below (in which $u$ is also required to be compatible with the maps $\exp : P \rightarrow \mathcal{O}_X$ and $\exp : Q \rightarrow \mathcal{O}_X$):
\begin{equation*} \xymatrix{
P \ar@{-->}[r]^u \ar[d] & Q \ar[d] \\
\o P \ar[r]^{\o u} & \o Q
} \end{equation*}
Replacing $Q$ with $\o{u}^{-1} Q$, we can assume that $\o P = \o Q$ and $\o u = \id_{\o P}$.%
\footnote{At this point a morphism $P \rightarrow Q$ covering $\o u$ must be an isomorphism so we could complete the proof using~\cite[Corollary~3.4]{Olsson-Log}.}

Let $H$ be the moduli space of maps $u : P \rightarrow Q$ that are compatible with $\o u = \id_{\o P}$, ignoring the maps to $\mathcal{O}_X$.  Locally such a map exists because $P$ and $Q$ are both extensions of $\o P$ by $\Gm$ and $\o P$ is generated by a finite collection of global sections.  Indeed, this implies that $P$ and $Q$ are each determined by a finite collection of $\Gm$-torsors on $X$, all of which can be trivialized after passage to a suitable open cover of $X$.  It follows that $H$ is a torsor on $X$ under $\Hom_{\Sch / X}(\o P, \Gm)$ and in particular is representable by an affine scheme over $X$.

We may now work relative to $H$ and assume that the map $u : P \rightarrow Q$ has already been specified.  We argue that the locus where the diagram
\begin{equation*} \xymatrix{
P \ar[dr]^\exp \ar[d] \\
Q \ar[r]_<>(0.5)\exp & \mathcal{O}_X
} \end{equation*}
commutes is closed.  In effect, we are looking at the locus where two $\Gm$-equivariant maps $P \rightarrow \mathbf{A}^1$ agree.  But $P$ is generated as a monoid with $\Gm$-action by a finite collection of sections, hence the agreement of the two maps $P \rightarrow \mathbf{A}^1$ corresponds to the agreement of a finite collection of pairs of sections of $\mathbf{A}^1$.  But $\mathbf{A}^1$ is separated, so this is representable by a closed subscheme.
\end{proof}

This completes the proof of Theorem~\ref{thm:log-hom}.
\end{proof}

We can now obtain a global variant:

\begin{corollary}
Let $X$ be a proper, flat, finite presentation algebraic space over $S$ and let $P$ and $Q$ be logarithmic structures on $X$ with $P$ coherent.  Then $\Hom_{\Sch / S}(P,Q)$ is representable by an algebraic space over $S$.
\end{corollary}
\begin{proof}
Let $\pi : X \rightarrow S$ be the projection.  Then we have
\begin{equation*}
\Hom_{\Sch / S}(P,Q) = \pi_\ast \Hom_{\Sch / X}(P,Q) .
\end{equation*}
We have already seen in Theorem~\ref{thm:log-hom} that $\Hom_{\Sch / X}(P,Q)$ is representable by an algebraic space over $X$ that is quasi-compact, quasi-separated, and locally of finite presentation.  We may therefore apply~\cite[Theorem~1.2]{HR-Hom} (or any of a number of other representability results for schemes of morphisms) to deduce the algebraicity of $\pi_\ast \Hom_{\Sch / X}(P,Q)$.
\end{proof}

\begin{corollary} \label{cor:log-tri-fixed-char}
Let $X$ and $S$ be as in the last corollary.  Suppose that $P$, $Q$, and $R$ are logarithmic structures on $X$ with $P$ and $Q$ coherent and morphisms $\alpha : P \rightarrow Q$ and $\beta : P \rightarrow R$ have been specified.  Then there is an algebraic space over $S$ parameterizing the commutative triangles shown below:
\begin{equation*} \xymatrix{
P \ar[d]_\alpha \ar[dr]^\beta \\
Q \ar[r] & R
} \end{equation*}
\end{corollary}
\begin{proof}
We recognize this functor as a fiber product:
\begin{equation*}
\Hom_{\Sch / S}(Q,R) \fpr_{\Hom_{\Sch/S}(P,R)} \{ \beta \}
\end{equation*}
\end{proof}

Proposition~\ref{prop:log-rep} is an immediate consequence of Corollary~\ref{cor:log-tri-fixed-char}, applied with $P = \pi^\ast M_S$, $Q = M$, and $R = M_X$.

\section{Local minimality}
\label{sec:local-case}

After Proposition~\ref{prop:log-rep}, all that is left to demonstrate Theorem~\ref{thm:main} is to verify Gillam's criteria for
\begin{equation*}
\Log(\Hom_{\LogSch/S}(M,M_X)) = \Hom_{\Sch/\Log(S)}(M,M_X).
\end{equation*}
As in the proof of Proposition~\ref{prop:log-rep} we separate this problem into local and global variants, the local version being the case $\u S = \u X$.  We treat the local problem in this section and deduce the solution to the global problem in the next one.

Let $\u X$ be a scheme equipped with three fine logarithmic structures, denoted $\pi^\ast M_S$, $M_X$, and $M$ in order to emphasize the application in the next section, and morphisms of logarithmic structures,
\begin{gather*}
\pi^\ast M_S \rightarrow M_X \\
\pi^\ast M_S \rightarrow M  .
\end{gather*}
Let $\GS^{\rm loc}(\u X)$ be the set of commutative diagrams
\begin{equation*} \xymatrix{
\pi^\ast M_S \ar[r] \ar[d] \ar@/^15pt/[rr] & N \ar[d] & M \ar[dl]^{\varphi} \\
M_X \ar[r] & N_X
} \end{equation*}
in which $N$ is a \emph{quasi-logarithmic structure}%
\footnote{The use of quasi-logarithmic structures here is entirely for convenience:  it allows us to avoid repeated passage to associated logarithmic structures.  The reader who would prefer not to worry about quasi-logarithmic structures should feel free to assume $N$ is a logarithmic structure and worry instead about remembering to take associated logarithmic structures at the right moments.}
(Definition~\ref{def:quasi-log}) and the square on the left is cocartesian.  These data are determined up to unique isomorphism by the quasi-logarithmic structure $N$, the morphism $\pi^\ast M_S \rightarrow N$, and the morphism $\varphi$.  We will refer to an object of $\GS^{\rm loc}(\u X)$ with the pair $(N, \varphi)$ with the morphism $\pi^\ast M_S \rightarrow N$ specified tacitly.%
\footnote{Effectively, $N$ is an object of the category of quasi-logarithmic structures equipped with a morphism from $\pi^\ast M_S$.}

\begin{convention} \label{conv:pushout}
As a matter of notation, whenever we have a morphism of monoids $\pi^\ast M_S \rightarrow N$ (resp.\ $\pi^\ast \o M_S \rightarrow \o N$), we write $N_X$ (resp. $\o N_X$) for the monoid obtained by pushout:
\begin{equation*} \vcenter{\xymatrix{
\pi^\ast M_S \ar[r] \ar[d] & M_X \ar[d] \\
N \ar[r] & N_X
}} \qquad \qquad \text{(resp.\ } \vcenter{\xymatrix{
\pi^\ast \o M_S \ar[r] \ar[d] & \o M_X \ar[d] \\
\o N \ar[r] & \o N_X
}} \text{)} \end{equation*}
When $N = N_S$ above, we simply write $N_X$ rather than $(N_S)_X$.
\end{convention}

The object of this section will be to prove the following two lemmas:

\begin{lemma} \label{lem:local-existence}
For any $\u X$-scheme $\u Y$, any object of $\GS^{\rm loc}(\u Y)$ admits a morphism from a minimal object.
\end{lemma}

\begin{lemma} \label{lem:local-pullback}
The pullback of a minimal object of $\GS^{\rm loc}(\u Y)$ via any morphism $\u Y' \rightarrow \u Y$ is also minimal.
\end{lemma}

The construction of the minimal object appearing in Lemma~\ref{lem:local-existence} is done in Section~\ref{sec:local}, while the proof of its minimality appears in Section~\ref{sec:Gillam}, along with the proof of Lemma~\ref{lem:local-pullback}.

\subsection{Construction of minimal objects}
\label{sec:local}

Fixing $(N,\varphi) \in \GS^{\rm loc}(\u X)$, we construct an object $(R, \rho) \in \GS^{\rm loc}(\u X)$ and a morphism $(R,\rho) \rightarrow (N,\varphi)$.  In Section~\ref{sec:Gillam} we verify that $(R,\rho)$ is minimal and that its construction is stable under pullback.

We assemble $R$ in steps:  First we build the associated group of its characteristic monoid, then we identify its characteristic monoid within this group, and finally we build the quasi-logarithmic structure above the characteristic monoid.  

Recall that the map $\varphi : M \rightarrow N_X$ induces a map $u : M \rightarrow N_X / \pi^\ast N_S \simeq M_X / \pi^\ast M_S$ known as the \emph{type} of $u$.  This generalizes \cite[Definition~1.10]{GS}.  For brevity, we write $\o M_{X/S} = M_X / \pi^\ast M_S = \o M_X / \pi^\ast \o M_S$ below.  Note that $u$ is equivariant with respect to the action of $\pi^\ast M_S$ on $M$ and the (trivial) action of $\pi^\ast M_S$ on the relative characteristic monoid $\o M_{X/S}$.  Therefore $u$ may equally well be considered a morphism $\o M / \pi^\ast \o M_S \rightarrow \o M_{X/S}$.

\begin{remark} \label{rem:hint}
The following construction is technical, so the reader may find it helpful to keep in mind that it is really an elaboration of an exercise in homological algebra:
\begin{quote}
\textit{If $0 \rightarrow A \rightarrow B \rightarrow C \rightarrow 0$ is an exact sequence of abelian groups, and $u : M \rightarrow C$ is a given homomorphism, there is a universal homomorphism $A \rightarrow A'$ such that $u$ lifts to a homomorphism $M \rightarrow B'$, where $B' = A' \mathbin{\amalg}_A B$:
\begin{equation*} \xymatrix{
& & & M \ar[d] \ar@{-->}[ddl] \\
0 \ar[r] & A \ar[r] \ar[d] & B \ar[r] \ar[d] & C \ar[r] \ar@{=}[d] & 0 \\
0 \ar[r] & A' \ar[r] & B' \ar[r] & C \ar[r] & 0
} \end{equation*}
Moreover, $A'$ may be taken to be $M \fpr_C B$.}
\end{quote}%
\end{remark}

\textsc{The associated group of the characteristic monoid of $R$.}  We set $\o R_0^{\rm gp} = \o M^{\rm gp} \fpr_{\o M_{X/S}^{\rm gp}} \o M_X^{\rm gp}$.  This fits into a commutative diagram with exact rows:
\begin{equation*} \xymatrix{
0 \ar[r] & \pi^\ast \o M_S^{\rm gp} \ar[r] \ar@{=}[d] & \o R_0^{\rm gp} \ar[r] \ar[d] & \o M^{\rm gp} \ar[d] \ar[r] & 0 \\
0 \ar[r] & \pi^\ast \o M_S^{\rm gp} \ar[r] & \o M_X^{\rm gp} \ar[r] & \o M_{X/S}^{\rm gp} \ar[r] & 0
} \end{equation*}
Observe that $\o R_0^{\rm gp}$ comes with \emph{two} maps $\pi^\ast \o M_S^{\rm gp} \rightarrow \o R^{\rm gp}$ corresponding to the two maps
\begin{gather*}
\pi^\ast \o M^{\rm gp}_S \rightarrow \o M^{\rm gp} \\
\pi^\ast \o M^{\rm gp}_S \rightarrow \o M_X^{\rm gp} .
\end{gather*}
We take $\epsilon : \o R_0^{\rm gp} \rightarrow \o R^{\rm gp}$ to be the quotient of $\o R_0^{\rm gp}$ by the diagonal copy of $\pi^\ast \o M_S^{\rm gp}$.  We may then define $\o R_X^{\rm gp}$ by pushout via $\pi^\ast \o M_S^{\rm gp} \rightarrow \o M_X^{\rm gp}$ (Convention~\ref{conv:pushout}).

\textsc{The homomorphism $\o\rho : \o M^{\rm gp} \rightarrow \o R_X^{\rm gp}$.}  Let $\o M_X^+$ be the pushout of $\o M_X$ by the homomorphism \emph{of monoids} $\pi^\ast \o M_S \rightarrow \pi^\ast \o M_S^{\rm gp}$.  This is a submonoid of $\o M_X^{\rm gp}$ and fits into the exact sequence in the middle row of the diagram below.  Let $\o R_0^+$ be the pullback of $\o M \rightarrow \o M_{X/S}$ to $\o M_X^+$ (the upper right square of the diagram below).  The diagram below is commutative except for the dashed arrows (which will be explained momentarily) and has exact rows:
\begin{equation*} \xymatrix{
0 \ar[r] & \pi^\ast \o M_S^{\rm gp} \ar[r] \ar[d] & \ar@{-->}[ddl]^(0.7){\epsilon} \o R_0^+ \ar[r] \ar[d]^{\beta} & \o M \ar[r] \ar[d] \ar@{-->}[ddl]^(0.7){\o\rho} & 0 \\
0 \ar[r] & \pi^\ast \o M_S^{\rm gp} \ar[r] \ar[d] & \o M_X^+ \ar[r] \ar[d] & \o M_{X/S} \ar[r] \ar@{=}[d] & 0 \\
0 \ar[r] & \o R^{\rm gp} \ar[r]_<>(0.5)\alpha & (\o R^{\rm gp})_X \ar[r] & \o M_{X/S}^{\rm gp} \ar[r] & 0
} \end{equation*}
Note that $(\o R^{\rm gp})_X$ is the pushout of $\pi^\ast \o M_S \rightarrow \o M_X$ via $\pi^\ast \o M_S \rightarrow \o R^{\rm gp}$ \emph{as a monoid}.  Equivalently, it is the pushout of $\pi^\ast \o M_S^{\rm gp} \rightarrow \o M_X^+$ via $\pi^\ast \o M_S^{\rm gp} \rightarrow \o R^{\rm gp}$, again as a monoid.  It is contained in but not necessarily equal to $\o R_X^{\rm gp} = (\o R_X)^{\rm gp}$.

The difference between the two compositions 
\begin{gather*}
\o R_0^+ \xrightarrow{\beta} \o M_X^+ \rightarrow (\o R^{\rm gp})_X \\
\o R_0^+ \xrightarrow{\epsilon} \o R^{\rm gp} \xrightarrow{\alpha} (\o R^{\rm gp})_X
\end{gather*}
factors uniquely through a map $\o M^{\rm gp} \rightarrow (\o R^{\rm gp})_X \subset \o R^{\rm gp}_X$.  We take this as the definition of $\o\rho$.

\begin{remark}
Observe that when the maps $\pi^\ast \o M_S^{\rm gp} \rightarrow \o R^{\rm gp}$ and $\pi^\ast \o M_S^{\rm gp} \rightarrow \o M^{\rm gp}$ are the canonical ones, the diagram on the left commutes but the diagram on the right does not:
\begin{equation*} \xymatrix{
\pi^\ast \o M^{\rm gp}_S  \ar[r] \ar[d] & \o M^{\rm gp} \ar[d]^{\o\rho} \\
\o R^{\rm gp} \ar[r]^\alpha & \o R^{\rm gp}_X
} \qquad \xymatrix{
\pi^\ast \o M^{\rm gp}_S  \ar[r] \ar[d] & \o M^{\rm gp} \ar[d]^{\o\rho} \\
\o R_0^{\rm gp} \ar[r] & (\o R_0^{\rm gp})_X
} \end{equation*}
This is the reason we introduced the quotient $\epsilon$ earlier.
\end{remark}

We view $(\o R^{\rm gp}, \o\rho)$ as the initial object of $\GS^{\rm loc}(\u X)$ \emph{on the level of associated groups of characteristic monoids}.  Justification for this attitude will be given in Section~\ref{sec:Gillam} (see the proof of Lemma~\ref{lem:initial-characteristic}).

\textsc{The characteristic monoid $\o R$.}  
We identify the smallest sheaf of submonoids $\o R \subset \o R^{\rm gp}$ that contains the image of $\pi^\ast \o M_S$ and whose pushout $\o R_X$ contains the image of $\o\rho : \o M \rightarrow \o R^{\rm gp}_X$.  For each local section $\xi$ of $\o M$ we will identify a local section (or possibly a finite collection of local sections) of $\o R^{\rm gp}$ for inclusion in $\o R$; we will then take $\o R$ to be the submonoid of $\o R^{\rm gp}$ generated by these local sections.  As $M$ is assumed to be coherent, a finite number of these local sections suffice to generate $\o R$, which guarantees that $\o R$ is coherent.

Suppose that $\xi \in \Gamma(U, \o M)$ is a section over some quasi-compact $U$ that is \'etale over $X$.  Recall that $\o\rho(\xi)$ lies in $(\o R^{\rm gp})_X$, which is the pushout of $\o R^{\rm gp}$ via the integral homomorphism $\pi^\ast \o M_S \rightarrow \o M_X$.  At least after passage to a finer quasi-compact \'etale cover, we can represent $\o\rho(\xi)$ as a pair $(a,b)$ where $a \in \o R^{\rm gp}$ and $b \in \o M_X$ (cf.\ Appendix~\ref{sec:integral}).

Let $B \subset \Gamma(U, \o M_X)$ be the collection of all $b \in \o M_X$ such that $\o\rho(\xi)$ can be represented as $(a,b)$ for some $a \in \Gamma(U, \o R^{\rm gp})$.  As $\o R^{\rm gp} \rightarrow \o R^{\rm gp}_X$ is injective (it is integral), there is at most one $a$ for any $b \in \Gamma(U, \o M_X)$.  Note that $B$ carries an action of $\Gamma(U, \pi^\ast \o M_S)$, for if $\o\rho(\xi)$ is representable by $(a,b)$ then it is also representable by $(a - c, b + c)$.  The action of the sharp monoid $\pi^\ast \o M_S$ gives $B$ a partial order by setting $b \leq b + c$ for all $c \in \Gamma(U, \pi^\ast \o M_S)$.  We will show $b$ has a least element with respect to this partial order.

Suppose $b$ and $b'$ are elements of $B$ with $(a,b)$ and $(a',b')$ both representing $\o\rho(\xi) \in \Gamma(U, \o R^{\rm gp}_X)$.  As $\pi^\ast \o M_S \rightarrow \o M_X$ is integral, Lemma~\ref{lem:integral-pushout} implies that there must be elements $d \in \o M_X$ and $c, c' \in \pi^\ast \o M_S$ with $a + c = a' + c'$ and $b = d + c$ and $b' = d + c'$.  But then $\o\rho(\xi)$ is also representable by $(a+c,d) = (a'+c',d)$.  Therefore, for any pair $b, b' \in B$ there is a $d \in B$ with $d \leq b$ and $d \leq b'$.

It will now follow that $B$ has a least element if we can show that every infinite decreasing chain of elements of $B$ stabilizes.  But $B$ is a subset of $\Gamma(U, \o M_X)$, and, at least provided $U$ has been chosen small enough, this is a strict submonoid of a finitely generated abelian group.  A strictly decreasing chain of elements must have strictly decreasing distance from the origin in $\Gamma(U, \o M_X) \tensor \mathbf{R}$, and there can be only a finite number of elements of $B$ whose image in $\Gamma(U, \o M_X) \tensor \mathbf{R}$ within a fixed distance of the origin.  Therefore the chain must stabilize and $B$ has a least element.

Writing $b$ for the least element of $B$ and $a$ for the corresponding element of $\Gamma(U,\o R^{\rm gp})$ such that $(a,b)$ represents $\o\rho(\xi)$, we include $a$ as an element of $\o R$.  As $\o M$ is coherent, we can repeat this construction for each element in a finite collection of sections that generate $\o M$ over $U$ (provided that $U$ has been chosen small enough).  

By construction, $\o R$ is locally of finite type and integral.  Moreover, the following lemma says that $(\o R, \o\rho)$ is the initial object of type $u$ in $\GS^{\rm loc}(\u X)$ \emph{on the level of characteristic monoids}.  We defer its proof to Section~\ref{sec:Gillam} in order not to interrupt the construction of $(R,\rho)$.

\begin{lemma} \label{lem:initial-characteristic}
For any $(N,\varphi) \in \GS^{\rm loc}(\u X)$ of type $u$ there is a unique morphism $(\o R, \o\rho) \rightarrow (\o N, \o\varphi)$.
\end{lemma}

\textsc{The quasi-logarithmic structure $R$ and the map $\rho$.}  This construction will require an object $(N,\varphi) \in \GS^{\rm loc}(\u X)$ and not just a type.  Suppose $(N,\varphi)$ has type $u$ and $(\o R, \o\rho)$ has been constructed as in the steps above.  Then Lemma~\ref{lem:initial-characteristic} implies that there is a canonical map $\o R \rightarrow \o N$ compatible with the tacit maps from $\pi^\ast \o M_S$.  By pulling back $N$ from $\o N$, we obtain a quasi-logarithmic structure $R$ with characteristic monoid $\o R$.

As we have a factorization (again by Lemma~\ref{lem:initial-characteristic})
\begin{equation*} 
\o M \xrightarrow{\o\rho} \o R_X \rightarrow \o N_X
\end{equation*}
and $R_X$ is pulled back from $N_X$ over $\o N_X$, the universal property of the fiber product yields an induced map $\rho : M \rightarrow R$.

\subsection{Verification of Gillam's criteria}
\label{sec:Gillam}

\begin{proof}[Proof of Lemma~\ref{lem:initial-characteristic}]
To prove the lemma, we must show that there is a unique morphism $\o\mu : \o R \rightarrow \o N$ such that the induced diagram below is commutative:
\begin{equation} \label{eqn:5} \vcenter{\xymatrix{
\o M \ar[r]^{\o\rho} \ar[dr]_{\o\varphi} & \o R_X \ar[d]^{\o\mu_X} \\
& \o N_X
}} \end{equation}


Given $(N,\varphi)$, we have a diagram with exact rows that commutes except for some of the parts involving the dashed arrow:
\begin{equation*} \xymatrix{
0 \ar[r] & \pi^\ast \o M_S^{\rm gp} \ar[r] \ar@{=}[d] & \o R_0^{\rm gp} \ar[r] \ar[d] & \o M^{\rm gp} \ar[r] \ar[d]^u \ar@{-->}[ddl]^(0.7){\o\varphi} & 0 \\
0 \ar[r] & \pi^\ast \o M_S^{\rm gp} \ar[r] \ar[d] & \o M_X^{\rm gp} \ar[r] \ar[d] & \o M_{X/S}^{\rm gp} \ar[r] \ar@{=}[d] & 0 \\
0 \ar[r] & \o N^{\rm gp} \ar[r] & \o N_X^{\rm gp} \ar[r] & \o M_{X/S}^{\rm gp} \ar[r] & 0
} \end{equation*}
The lower triangle involving $\o\varphi$ commutes (because $(N,\varphi)$ has type $u$) but the upper one may not.  The difference of the two compositions
\begin{gather*}
\o R_0^{\rm gp} \rightarrow \o M^{\rm gp} \xrightarrow{\o\varphi} \o N_X^{\rm gp} \\
\o R_0^{\rm gp} \rightarrow \o M_X^{\rm gp} \rightarrow \o N_X^{\rm gp}
\end{gather*}
factors uniquely through a map $\o \mu : \o R_0^{\rm gp} \rightarrow \o N^{\rm gp}$.  Moreover, $\o\mu$ vanishes on the diagonal copy of $\pi^\ast \o M_S^{\rm gp}$ inside $\o R_0^{\rm gp}$.  This gives a factorization:
\begin{equation*}
\pi^\ast \o M_S^{\rm gp} \rightarrow \o R^{\rm gp} \xrightarrow{\o\mu} \o N^{\rm gp} 
\end{equation*}
Moreover, the construction of $\o\mu$ is easily reversed to give a bijective correspondence between maps $\o\mu : \o R^{\rm gp} \rightarrow \o N^{\rm gp}$ compatible with the tacit maps from $\pi^\ast \o M_S^{\rm gp}$ and morphisms $\o\varphi : \o M^{\rm gp} \rightarrow \o N_X^{\rm gp}$ compatible with the type (cf.\ Remark~\ref{rem:hint}).

We must verify that the image of $\o\mu : \o R \rightarrow \o N^{\rm gp}$ lies in $\o N$.  By definition, $\o N_X \subset \o N_X^{\rm gp}$ contains the image of $\o\varphi : \o M \rightarrow \o N_X^{\rm gp}$.  Write $\o R' \subset \o R^{\rm gp}$ for the pre-image of $\o N$ via the map $\o R^{\rm gp} \rightarrow \o N^{\rm gp}$; thus $\o R' \rightarrow \o N$ is an exact morphism of monoids \cite[Definition~4.6~(1)]{Kato}.  Then $\o R'_X \rightarrow \o N_X$ is also exact [\emph{loc.\ cit.}], so $\o R'_X$ coincides with the pre-image of $\o N_X \subset \o N_X^{\rm gp}$.  Furthermore, $\o R'_X$ contains the image of $\o\rho : \o M \rightarrow \o R_X^{\rm gp}$ because $\o N_X$ contains the image of $\o\varphi : \o M \rightarrow \o N_X^{\rm gp}$.  On the other hand, $\o R$ was constructed as the smallest submonoid of $\o R^{\rm gp}$ such that $\o R_X$ contains the image of $\o\rho : \o M \rightarrow \o R_X^{\rm gp}$.  Therefore $\o R \subset \o R'$ and the image of $\o R \rightarrow \o N_X$ is contained in $\o N$.
 
Finally, to get the commutativity of~\eqref{eqn:5}, it is sufficient to work on the level of associated groups.  Assemble the diagram below, which is commutative except for some of the parts involving the dashed arrows:
\begin{equation*} \xymatrix{
0 \ar[r] & \pi^\ast \o M_S^{\rm gp} \ar[r] \ar[d] & \o R_0^{\rm gp} \ar[r] \ar[d]^(0.4){\beta} \ar@{-->}[dl]_{\epsilon} \ar@{-->}[ddl]^(0.7){\o\mu \epsilon} & \o M^{\rm gp} \ar[r] \ar[d] \ar@{-->}[ddl]^(0.7){\o\varphi} \ar@{-->}[dl]_{\o\rho} & 0 \\
0 \ar[r] & \o R^{\rm gp} \ar[r] \ar[d]_{\o\mu} \ar[r] & \o R_X^{\rm gp} \ar[r] \ar[d]_(0.5){\o\mu_X} & \o M_{X/S}^{\rm gp} \ar[r] \ar@{=}[d] & 0 \\
0 \ar[r] & \o N^{\rm gp} \ar[r] & \o N_X^{\rm gp} \ar[r] & \o M_{X/S}^{\rm gp} \ar[r] & 0
} \end{equation*}
We write $d$ for any of the horizontal arrows.  To show that $\o\varphi = \o\mu_X \circ \o\rho$ it is sufficient to show that $\o\varphi \circ d = \o\mu_X \circ \o\rho \circ d$.  Recall that $\o\rho$ was constructed such that $\o\rho \circ d = \beta - d \circ \epsilon$ and $\o\mu$ was constructed such that $d \circ \o\mu \circ \epsilon = \o\mu_X \circ \beta - \o\varphi \circ d$.  Therefore,%
\begin{equation*}
\o\mu_X \circ \o\rho \circ d  = \o\mu_X \circ \beta - \o\mu_X \circ d \circ \epsilon  = \o\mu_X \circ \beta - d \circ \o\mu \circ \epsilon = \o\varphi \circ d
\end{equation*}
as required.
\end{proof}

\begin{proof}[Proof of Lemma~\ref{lem:local-existence}]
To minimize excess notation, we assume (without loss of generality) that $\u Y = \u X$.

Consider a morphism $(N',\varphi') \rightarrow (N,\varphi)$ of $\GS^{\rm loc}(\u X)$.  We verify that the map $(R,\rho) \rightarrow (N,\varphi)$ factors in a unique way through $(N',\varphi')$.  In Lemma~\ref{lem:initial-characteristic}, we have already seen that there is a unique map $\o R \rightarrow \o N'$, compatible with the maps from $\pi^\ast \o M_S$, and a unique factorization of $\o\varphi : \o M \rightarrow \o N'_X$ through $\o\rho : \o M \rightarrow \o R_X$.  In particular, the diagram below commutes:
\begin{equation*} \xymatrix{ 
\o R \ar[r]^{\o\varphi'} \ar[dr]_{\o\varphi} & \o N' \ar[d] \\
& \o N
} \end{equation*}
Since $N'$ is pulled back from $N$ by the vertical arrow in the diagram above, this gives a uniquely determined arrow $R \rightarrow N'$.  Likewise, the diagrams of solid arrows below are commutative (the diagonal arrow on the left coming from Lemma~\ref{lem:initial-characteristic}):
\begin{equation*} \xymatrix{
\o M \ar[r] \ar[d] & \o R_X \ar[dl] \ar[d] \\
\o N_X' \ar[r] & \o N_X
} \qquad \xymatrix{
M \ar[r] \ar[d] & R_X \ar[d] \ar@{-->}[dl] \\
N_X' \ar[r] & N_X
} \end{equation*}
As $N'_X$ is pulled back from $N_X$ via the map $\o N'_X \rightarrow \o N_X$ there is therefore a unique induced map $R \rightarrow N'_X$ completing the diagram on the right.  This proves the minimality of $(R,\rho)$.
\end{proof}

\begin{proof}[Proof of Lemma~\ref{lem:local-pullback}]
We show that the tools used in the construction of $R$ and $\rho$ are all compatible with pullback of quasi-logarithmic structures.  Pullback of quasi-logarithmic structures along a morphism $g : Y' \rightarrow Y$ involves two steps:  pullback of \'etale sheaves along $g$ followed by pushout of extensions along $g^{-1} \mathcal{O}_Y^\ast \rightarrow \mathcal{O}_{Y'}^\ast$.  We verify that the construction of $R$ and $\rho$ commutes with these operations:
\begin{enumerate}
\item $\o R^{\rm gp}$ was constructed as a quotient of a fiber product, and both fiber products and quotients are preserved by pullback of \'etale sheaves;
\item $\o\rho$ was induced by the universal property of $\o M$ as a quotient, and quotients are preserved by pullback of \'etale sheaves;
\item $\o R$ was built as the sheaf of submonoids of $\o R^{\rm gp}$ generated by a collection of local sections, and this construction is compatible with pullback of \'etale sheaves;
\item $R$ was constructed as a pullback of an extension by $\mathcal{O}_Y^\ast$, and such pullbacks are preserved by pullback of \'etale sheaves and by pushout along $g^{-1} \mathcal{O}_Y^\ast \rightarrow \mathcal{O}_{Y'}^\ast$;
\item $\rho$ was induced by the universal property of a base change of extensions, and, as remarked above, base change of extensions is preserved by pullback of \'etale sheaves and pushout along the kernels.
\end{enumerate}
\end{proof}

\section{Global minimality}
\label{sec:global-case}

In this section, $X = (\u X, M_X)$ and $S = (\u S, M_S)$ will be fine logarithmic schemes.  We assume that the projection $\pi : X \rightarrow S$ is proper and flat \emph{with reduced geometric fibers} and that the morphism of logarithmic structures $\pi^\ast M_S \rightarrow M_X$ is integral.  We also assume a second coherent logarithmic structure $M$ on $X$ has been specified, along with a morphism $\pi^\ast M_S \rightarrow M$.  
We will verify Gillam's minimality criterion (Proposition~\ref{prop:minimal-criteria}) for the fibered category $\Hom_{\LogSch/S}(M,M_X)$ over $\LogSch$. 

We continue to use Convention~\ref{conv:pushout} to notate pushouts, as well as to work with quasi-logarithmic structures instead of logarithmic structures.  Define $\GS(S)$ to be the category of pairs $(N_S,\varphi)$ where $N_S$ is a quasi-logarithmic structure on $S$ equipped with a tacitly specified morphism $M_S \rightarrow N_S$ and $\varphi$ fits into a commutative diagram
\begin{equation*} \xymatrix{
\pi^\ast M_S \ar[r] \ar[d] \ar@/^15pt/[rr] & \pi^\ast N_S \ar[d] & M \ar[dl]^{\varphi} \\
M_X \ar[r] & N_X
} \end{equation*}
whose square is cocartesian.  By pullback of quasi-logarithmic structures, we may assemble this definition into a fibered category over $\LogSch/S$.  When $T$ is strict over $S$, we write $\GS(\u T)$ instead of $\GS(T)$.

We may recognize $\Hom_{\LogSch/S}(M,M_X)(T)$%
\footnote{The notation $\Hom_{\LogSch/S}(M,M_X)$ refers to the fibered category over $\LogSch/S$ whose value on a logarithmic scheme $T$ over $S$ is the set of morphisms of logarithmic structures $M \rest{T} \rightarrow M_X \rest{T}$ on $\u X \mathbin{\times}_{\u S} \u T$, compatible with the tacit morphisms from $\pi^\ast M_T$.}
inside of $\GS(T)$ as the opposite of the category of pairs $(N_{T},\varphi)$ where $N_{T}$ is a logarithmic structure, as opposed to merely quasi-logarithmic structure.  We are free to work with $\GS(T)$ in place of $\Hom_{\LogSch/S}(M,M_X)(T)$ as minimal objects of the latter may be induced from minimal objects of the former by passage to the associated logarithmic structure.
  
\begin{lemma} \label{lem:minimal-existence}
For any logarithmic scheme $T$ over $S$, every object of $\GS(T)$ admits a morphism from a minimal object.
\end{lemma}

\begin{lemma} \label{lem:minimal-pullback}
If $f : T' \rightarrow T$ is a morphism of logarithmic schemes over $S$ and $(Q_T,\psi)$ is minimal in $\GS(T)$ then $f^\ast (Q_T, \psi)$ is minimal in $\GS(T')$.
\end{lemma}

Note that passage to the associated logarithmic structure commutes with pullback of pre-logarithmic structures \cite[(1.4.2)]{Kato}, so Lemma~\ref{lem:minimal-pullback} implies that minimal objects of $\Hom_{\LogSch/S}(M,M_X)(\u T)$ pull back to minimal objects of $\Hom_{\LogSch/S}(M,M_X)(\u T')$.

The strategy of proof for Lemmas~\ref{lem:minimal-existence} and~\ref{lem:minimal-pullback} will be to bootstrap from the minimal objects of $\GS^{\rm loc}(\u X)$ constructed in Section~\ref{sec:local-case}.  The essential tool in this construction is a left adjoint to pullback for \'etale sheaves, constructed in Section~\ref{sec:pi_0}, for whose construction is the reason we must assume $X$ has reduced geometric fibers over $S$.  Having dispensed with generalities in Section~\ref{sec:pi_0}, we take up the construction of minimal objects of $\GS(\u T)$ in Section~\ref{sec:reduction}.

\subsubsection*{Zariski logarithmic structures}

The following proposition will only be used in Appendix~\ref{app:calc}.  It shows that when $M_S$ and $M$ are Zariski logarithmic structures, one can replace $M_X$ by its best approximation by a Zariski logarithmic structure for the purpose of constructing the category $\GS$.  At least in many situations, this means that one can work in the Zariski topology rather than the \'etale topology for the purpose of constructing a minimal logarithmic structure.  See Appendix~\ref{app:calc} for more details.

\begin{lemma} \label{lem:zar-pull}
	Let $X$ be a scheme.  Denote by $\tau : \et(X) \rightarrow \zar(X)$ the morphism of sites.  Then $\Hom(\tau^\ast F, \tau^\ast G) = \Hom(F,G)$ for any sheaves $F$ and $G$ on $\zar(X)$.  In particular, $\tau_\ast \tau^\ast \simeq \id$.
\end{lemma}
\begin{proof}
	By adjunction, we have $\Hom(\tau^\ast F, \tau^\ast G) = \Hom(F, \tau_\ast \tau^\ast G)$.  But we can calculate that $\tau_\ast \tau^\ast G(U) = \tau^\ast G(U) = G(U)$ for any open $U \subset X$.  The first equality is the definition; the second equality holds because, for example, the espace \'etal\'e of $G$ (in the Zariski topology) is a scheme, hence satisfies \'etale descent.
\end{proof}

\begin{proposition} \label{prop:zar-GS}
	Suppose that $M$ and $M_S$ are Zariski logarithmic structures and that $\o M_S^{\rm gp}$ is torsion free.  Let $\tau$ denote the canonical morphism from the \'etale site to the Zariski site.  Define $\GS'$ be the category obtained by imitating the definition of $\GS$ with $M_X$ replaced by $\tau^\ast \tau_\ast M_X$.  Then $\GS \simeq \GS'$.
\end{proposition}
\begin{proof}
	By assumption, we have $M_S = \tau^\ast M'_S$ and $M = \tau^\ast M'$ for some logarithmic strcuctures $M'_S$ on $S$ and $M'$ on $\zar(X)$.  Here $\tau^\ast$ denotes pullback of logarithmic structures.  
	
	We will begin by constructing a functor $\GS' \rightarrow \GS$.  Observe that an object of $\GS'$ is a commutative diagram, in which the square is cocartesian:

	\begin{equation*} \xymatrix{
			\pi^\ast M_S \ar[r] \ar[d] \ar@/^15pt/[rr] & \tau^\ast \tau_\ast M_X \ar[d] & M \ar[dl] \\
			\pi^\ast N_S \ar[r] & N'_X 
	} \end{equation*}
	Then composition with $\tau^\ast \tau_\ast M_X \rightarrow M_X$ induces
	\vskip1ex
	\begin{equation*} \xymatrix{
			\pi^\ast M_S \ar[r] \ar[d] \ar@/^20pt/[rrr] & \tau^\ast \tau_\ast M_X \ar[r] \ar[d] & M_X \ar[d] & M \ar[dl] \\
			\pi^\ast N_S \ar[r] & N'_X \ar[r] & N_X
	} \end{equation*}
	and omitting $\tau^\ast \tau_\ast M_X$ and $N'_X$ yields an object of $\GS$.

	Now we construct the functor $\GS \rightarrow \GS'$.  Suppose that we have an object of $\GS$:

	\begin{equation*} \xymatrix{
			\pi^\ast M_S \ar[r] \ar[d] \ar@/^15pt/[rr] & M_X \ar[d] & \tau^\ast M \ar[dl] \\
			\pi^\ast N_S \ar[r] & N_X 
	} \end{equation*}
	Applying $\tau_\ast$, this gives the following diagram:

	\begin{equation} \label{eqn:11} \vcenter{ \xymatrix{
				\pi^\ast_{\zar} M'_S \ar[r] \ar[d] \ar@/^15pt/[rr] & \tau_\ast M_X \ar[d] & M' \ar[dl] \\
				\tau_\ast \pi^\ast N_S \ar[r] & \tau_\ast N_X 
	} } \end{equation}
	Note that we have used $\tau_\ast \pi^\ast M_S = \tau_\ast \tau^\ast \pi_{\zar}^\ast M'_S = \pi_{\zar}^\ast M'_S$ and $\tau_\ast M = \tau_\ast \tau^\ast M' = M'$ by Lemma~\ref{lem:zar-pull}.  The square in the diagram above is cocartesian.  Indeed, first construct an exact sequence:
	\begin{equation*}
		0 \rightarrow \pi^\ast M_S^{\rm gp} \rightarrow (\pi^\ast N_S \times M_X)^{\sim} \rightarrow N_X \rightarrow 0
	\end{equation*}
	where $(\pi^\ast N_S \times M_S)^{\sim}$ is the smallest submonoid of $\pi^\ast N_S^{\rm gp} \times M_X^{\rm gp}$ that contains $\pi^\ast N_S \times M_X$ and the image of $\pi^\ast M_S^{\rm gp}$.  Applying $\tau_\ast$ to this gives an exact sequence:
	\begin{equation} \label{eqn:13}
		0 \rightarrow \tau_\ast \pi^\ast M_S^{\rm gp} \rightarrow \tau_\ast (\pi^\ast N_S \times M_X)^{\sim} \bigr) \rightarrow \tau_\ast N_X \rightarrow R^1 \tau_\ast \pi^\ast M_X^{\rm gp}
	\end{equation}
	The notation in the middle term is unabiguous because
	\begin{equation*}
		\bigl( \tau_\ast (\pi^\ast N_S \times M_X) \bigr)^\sim = \tau_\ast \bigl( (\pi^\ast N_S \times M_X )^\sim \bigr) 
	\end{equation*}
	via the natural map.  To see this, note first that it is sufficient to verify this at the level of characteristic monoids, since both sides are torsors under $\mathcal O_X^\ast$ over their characteristic monoids (the pushforward of a $\mathcal O_X^\ast$-torsor being a $\mathcal O_X^\ast$-torsor by Hilbert's Theorem 90).  It is also sufficient to check this on stalks, so we may assume that $X$ is the spectrum of a field, and in particular that $\o M'_S$ is constant.  An element of $\tau_\ast \bigl( (\pi^\ast \o N_S \times \o M_X )^\sim \bigr)$ is then a section of $\pi^\ast \o N_S^{\rm gp} \times \o M_X^{\rm gp}$ that can be expressed as $x - y$ for some $x \in \pi^\ast \o N_S \times \o M_X$ and $y \in \o M_S$ and is invariant under the action of the Galois group.  An element of $\bigl( \tau_\ast (\pi^\ast \o N_S \times \o M_X) \bigr)^\sim$ is of the form $x - y$ where $x$ is a Galois invariant section of $\pi^\ast \o N_S \times \o M_X$ and $y$ is a section of $\o M'_S$.  But the Galois action on $\o M_S$ is trivial since $\o M_S = \tau^\ast \o M'_S$, so $x - y$ is Galois invariant if and only if $x$ is.

	Now we show that $R^1 \tau_\ast \pi^\ast M_S^{\rm gp} = 0$.  We can verify this by passing to stalks an assume that $X$ is the spectrum of a field.  Note that $\pi^\ast M_S^{\rm gp}$ is an extension of a torsion free abelian group by $\mathcal O_X^\ast$.  We know that $R^1 \tau_\ast \mathcal O_X^\ast = 0$ by Hilbert's Theorem 90, and $R^1 \tau_\ast \pi^\ast \o M_S^{\rm gp} = 0$ because we can identify it with homomorphisms from the Galois group, which is profinite, into the discrete, torsion free abelian group $\o M_S^{\rm gp}$.

	Now the exact sequence~\eqref{eqn:13} implies that the square in diagram~\eqref{eqn:11} is cocartesian.  Applying $\tau^\ast$ to diagram~\eqref{eqn:11} we get an object of $\GS'$:
	\begin{equation*} \xymatrix{
			\pi^\ast M_S \ar[r] \ar@/^15pt/[rr] \ar[d] & \tau^\ast \tau_\ast M_X \ar[d] & M \ar[dl] \ar@/^10pt/[ddl] \\
			\tau^\ast \tau_\ast \pi^\ast N_S \ar[r] \ar[d] & \tau^\ast \tau_\ast N_X \ar[d] \\
			\pi^\ast N_S \ar[r] & N'_X
	} \end{equation*}
	Here $N'_X$ is defined to make the bottom square cartesian.  But pullback preserves cartesian diagrams, so both squares are cartesian and the outer part of the diagram is the desired object of $\GS'$.  We leave it to the reader to verify that these constructions are inverse to one another.
\end{proof}

\subsection{Left adjoint to pullback}
\label{sec:pi_0}

In this section we prove that pullback for \'etale sheaves has a left adjoint under two natural hypotheses (flatness and local finite presentation) and one apparently unnatural one (reduced geometric fibers).  When $f : X \rightarrow S$ is \'etale, the left adjoint to pullback exists for obvious reasons and is well known:  simply compose with the structure morphism of the espace \'etal\'e with $f$.  The construction in the present generality appears to be new.

Our construction is based on the following observation:  If $F$ is an \'etale sheaf on $X$, write $F^{\et}$ for its espace \'etal\'e.  When $S$ is the spectrum of a separably closed field then $f_! F$ has no choice but to be the set of connected components of $F^{\et}$.  In general, if the definition of $f_!$ is to be compatible with base change in $S$, this forces $f_! F$ to coincide with $\pi_0(F^{\et} / S)$, as defined by Laumon and Moret--Bailly~\cite[Section~(6.8)]{LMB} or Romagny~\cite{Romagny}.

The results of \cite{LMB} and \cite{Romagny} guarantee the existence of $\pi_0(F^{\et}/S)$ as long as $F^{\et}$ is flat, of finite presentation, and possesses reduced geometric fibers over $S$.%
\footnote{In fact, \cite{LMB} assumes that $X$ is smooth over $S$, but, as we will see below, only flatness is necessary in the construction.}
This suffices to treat a large enough class of \'etale sheaves to generate all others under colimits when $X$ is merely locally of finite presentation over $S$.  As $f_!$ must respect colimits where defined, we can then extend the definition to any \'etale sheaf $F$ by applying $f_!$ to a diagram of \'etale sheaves over $X$ with colimit $F$ and then taking the colimit of the resulting sheaves over $S$.

Flatness and local finite presentation appear to be natural hypotheses for the existence of $f_!$ in the sense that removing either leads immediately to counterexamples (the inclusion of a closed point or the spectrum of a local ring, respectively).  It is less clear how essential it is to require reduced geometric fibers, as our construction makes use of that hypothesis only to ensure the existence of $\pi_0(F^{\et}/S)$ as an \'etale sheaf.

Gabber has argued~\cite{Gabber} that the natural condition on $f$ under which $f^\ast$ possesses a left adjoint is, in addition to suitable finiteness conditions, that the morphism $\widetilde{X} \rightarrow \widetilde{S}$ possess connected geometric fibers whenever $\widetilde{X}$ is the strict henselization of $X$ at a geometric point $x$ and $\widetilde{S}$ is the strict henselization of $S$ at $f(x)$.

\begin{theorem} \label{thm:left-adj}
Suppose that $f : X \rightarrow S$ is a flat, local finite presentation morphism of algebraic spaces with reduced geometric fibers.  Then the functor $f^\ast : \et(S) \rightarrow \et(X)$ on \'etale sheaves has a left adjoint, $f_!$.
\end{theorem}
\begin{proof}
In this proof we will move freely between \'etale sheaves and their espaces \'etal\'es, which are algebraic spaces that are \'etale over the base.  The \'etale site $\et(-)$ will be taken to mean the category of all \'etale \emph{algebraic spaces} over the base, so that it coincides with the category of \'etale sheaves.

For any $U \in \et(X)$, we may define a functor:
\begin{equation*}
F_U : \et(S) \rightarrow \Sets : V \mapsto \Hom_S(U,V) = \Hom_X(U, f^\ast V)
\end{equation*}
Consider the collection $\sC$ of all $U \in \et(X)$ for which $F_U$ is representable by an \'etale sheaf on $S$.  The existence of $f_!$ is equivalent to the assertion that $\sC = \et(X)$.

\begin{enumerate}[label=\textsc{Step} \arabic{*}.,ref=\textsc{Step} \arabic{*}, labelindent=0pt, leftmargin=0pt,labelsep=1em,parsep=0pt,itemindent=*,listparindent=\parindent,itemsep=5pt]
\item \label{left-adj:step:1} We observe first that $\sC$ is closed under colimits:  Suppose that $U = \varinjlim U_i$ and $F_{U_i}$ is representable by $f_! U_i$ for all $i$.  Then we may take $f_! U = \varinjlim f_! U_i$:
\begin{align*}
F_{U}(V) & 
= \Hom_{\et(X)}(\varinjlim U_i, f^\ast V) \\
& = \varprojlim \Hom_{\et(X)}(U_i, f^\ast V) \\
& = \varprojlim \Hom_{\et(S)}(f_! U_i, V) \\
& = \Hom_{\et(S)}(\varinjlim f_! U_i, V)
\end{align*}

Every \'etale algebraic space over $X$ is a colimit of \'etale algebraic spaces of finite presentation over $S$.  For example, every \'etale algebraic space over $X$ is a colimit of \'etale algebraic spaces that are affine over affine open subsets of $S$.  Therefore \ref{left-adj:step:1} implies that the construction of $f_! F$ for arbitrary \'etale sheaves $F$ on $X$ reduces to the construction for those representable by algebraic spaces of finite presentation over $S$.

\item \label{left-adj:step:2} Following the construction of $\pi_0(X/S)$ from \cite[Section~(6.8)]{LMB}, we argue next that $\sC$ contains all \'etale $U$ over $X$ that are of finite presentation over $S$.  One could also use the construction of $\pi_0(X/S)$ from \cite[Th\'eor\`eme~2.5.2~(i)]{Romagny}.

Suppose that $U$ is flat, of finite presentation, and representable by schemes over $S$.  By \cite[Corollaire~(15.6.5)]{ega4-3}, there is an open subscheme $W \subset U \fpr_S U$ such that, for each point $x$ of $U$, the open set $W \cap (\{ x \} \times U) \subset U$ is the connected component of $x$ in $U$.  Thus a field-valued point of $U \fpr_S U$ lies in $W$ if and only if its two projections to $U$ lie in the same connected component.  Thus $W \subset U \fpr_S U$ is a flat equivalence relation on $U$, hence has a quotient $\pi_0(U/S) = U/W$ that is an algebraic space over~$S$.

We verify that $U/W$ is \'etale over $S$.  It is certainly flat and locally of finite presentation since $U$ is.  It is therefore enough to verify it is formally unramified.  This condition can be verified after base change to the geometric points of $S$.  As the definition of $W$ commutes with base change, so does the quotient $U/W$.   We can therefore assume $S$ is the spectrum of a separably closed field and then $U/W = \pi_0(U/S) = \pi_0(U)$ is simply the set of connected components of $U$, which is certainly unramified over $S$.

Now we show that $\pi_0(U/S)$ represents $F_U$.  If $g : U \rightarrow V$ is a morphism from $U$ to an \'etale $S$-scheme $V$ then the pre-images of points of $V$ are open and closed in their fibers over $S$ (since $V$ has discrete fibers over $S$).  Therefore, $U \fpr_V U$ contains $W$ (viewing both as open subschemes of $U \fpr_S U$), so $f$ factors through $U/W$.
\end{enumerate}
\end{proof}

\begin{corollary} \label{cor:pi_!-pullback}
Let $f : X \rightarrow S$ be as in the statement of the theorem and let $f' : X' \rightarrow S'$ be deduced by base change via a morphism $g : S' \rightarrow S$.  Then the natural morphism $f'_! g^\ast \rightarrow g^\ast f_!$ is an isomorphism.
\end{corollary}
\begin{proof}
Since a morphism of \'etale sheaves is an isomorphism if and only if it is an isomorphism on stalks, it is sufficient to verify the assertion upon base change to all geometric points of $S'$ and therefore to assume that $S'$ is itself a geometric point.  Since every \'etale sheaf is a colimit of representable \'etale sheaves that are of finite presentation over $S$ (as in the proof of Theorem~\ref{thm:left-adj}), it is sufficient to show that
\begin{equation*}
f'_! g^\ast F \rightarrow g^\ast f_! F
\end{equation*}
when $F$ is representable by a scheme that is of finite presentation over $S$.  In that case, $f_! F = \pi_0(F^{\et}/S)$ and $f'_! g^\ast F = \pi_0(F^{\et} \fpr_S S' /S')$.  But the fiber of $\pi_0(F^{\et}/S)$ over $S'$ is $\pi_0(F^{\et} \fpr_S S' / S')$ by definition!
\end{proof}

The following proposition is well-known and included only for completeness.

\begin{proposition}
Let $X$ be a site.  The inclusion of sheaves of abelian groups (resp.\ sheaves of commutative monoids) on $X$ in sheaves of sets on $X$ admits a left adjoint $F \mapsto \mathbf{Z} F$ (resp.\ $F \mapsto \mathbf{N} F$).
\end{proposition}
\begin{proof}
The proofs for abelian groups and for commutative monoids are identical, so we only write the proof explicitly for abelian groups.

It is equivalent to demonstrate that, for any sheaf of sets $F$ on $X$ there is an initial pair $(G, \varphi)$ where $G$ is a sheaf of abelian groups and $\varphi : F \rightarrow G$ is a morphism of sheaves of sets.  Denote the category of pairs $(G, \varphi)$ by $\sC$.  By the adjoint functor theorem, $\sC$ has an initial object if it is closed under small limits and has an essentially small coinitial subcategory \cite[Theorem~X.2.1]{cftwm}.  Closure under small limits is immediate.  

For the essentially small coinitial subcategory, take the collection $\mathscr{C}_0$ of all $(G, \varphi)$ such that $\varphi(F)$ generates $G$ as a sheaf of abelian groups (i.e., the smallest subsheaf of abelian groups $G' \subset G$ that contains $\varphi(F)$ is $G$ itself).  The cardinalities of $G'(U)$ for all $U$ in a set of topological generators may be bounded in terms of the cardinalities of the $F(U)$.  It follows that $\mathscr{C}_0$ is essentially small and by the adjoint functor theorem that the inclusion of sheaves of abelian groups in sheaves of sets has a left adjoint.
\end{proof}

\begin{proposition}
	Let $f : X \rightarrow S$ be flat and locally of finite presentation with reduced geometric fibers.  The functor $f^\ast$ on sheaves of abelian groups (resp.\ sheaves of commutative monoids) has a left adjoint, denoted $f_!$.
\end{proposition}
\begin{proof}
The proof is the same as the proof of the theorem above.  We recognize that the class of sheaves of abelian groups (resp.\ sheaves of commutative monoids) $F$ for which $f_! F$ exists is closed under colimits.  It contains $\mathbf{Z} U$ (resp.\ $\mathbf{N} U$) for all \'etale $U$ over $X$ since we may take
\begin{equation*}
f_! (\mathbf{Z} U) = \mathbf{Z} f_!(U) \qquad \text{(resp.\ } f_!(\mathbf{N} U) = \mathbf{N} f_!(U) \text{)}.
\end{equation*}
Finally, all sheaves of abelian groups are colimits of diagrams of $\mathbf{Z} U$ (resp.\ $\mathbf{N}U$) as above, so $f_!$ is defined for all sheaves of abelian groups on $X$.
\end{proof}

\subsubsection*{Zariski sheaves}

We include a statement about the left adjoint to pullback on sheaves in the Zariski topology in a restricted situation when it agrees with the left adjoint on \'etale sheaves.  In practice, this can be used to compute the left adjoint on \'etale sheaves by working in the Zariski topology, as in Appendix~\ref{app:calc}.

\begin{proposition} \label{prop:zar-push}
	Let $S$ be the spectrum of an algebraically closed field and let $f : X \rightarrow S$ be a reduced, finite type scheme over $S$.  Then $f^\ast : \zar(S) \rightarrow \zar(X)$ has a left adjoint, given by $f_! \tau^\ast$, where $f_!$ denotes the left adjoint on \'etale sheaves.
\end{proposition}
\begin{proof}
	As in Lemma~\ref{lem:zar-pull}, we write $\tau$ for the morphism from the \'etale site to the Zariski site.  We have 
	\begin{multline*}
		\Hom(f_! \tau^\ast F, G) = \Hom(\tau^\ast F, f^\ast_{\et} G) \\ = \Hom(\tau^\ast F, \tau^\ast f^\ast_{\zar} G) = \Hom(F, f^\ast_{\zar} G)
	\end{multline*}
	as required.
\end{proof}

\subsection{Reduction to the local problem}
\label{sec:reduction}

It is sufficient to construct minimal objects in $\GS(S)$.  When $(N_S,\varphi)$ is an object of $\GS(\u S)$, the pair $(\pi^\ast N_S,\varphi)$ is an object of $\GS^{\rm loc}(\u X)$.  This determines a functor $\GS(\u S) \rightarrow \GS^{\rm loc}(\u X)$, and when $\pi : \u X \rightarrow \u S$ is an isomorphism, this functor induces an equivalence between $\GS(\u S)$ and $\GS^{\rm loc}(\u X)$.

Before giving the proof Lemma~\ref{lem:minimal-existence}, we explain the construction.  In order to minimize notation, we construct minimal objects of $\GS(\u S)$; the same construction applies to $\GS(\u T)$ for any $\u S$-scheme $\u T$ after pulling back the relevant data.

We suppose that $(N_S,\varphi)$ is an object of $\GS(\u S)$ and we construct a pair $(Q_S,\psi)$ in $\GS(\u S)$ and a morphism $(Q_S,\psi) \rightarrow (N_S,\varphi)$.  Then we show $(Q_S,\psi)$ is minimal in $\GS(\u S)$ and that the construction of $(Q_S,\psi)$ is stable under pullback.

Let $(R,\rho) \rightarrow (\pi^\ast N_S,\varphi)$ be a morphism from a minimal object of $\GS^{\rm loc}(\u X)$, as guaranteed by Lemma~\ref{lem:local-existence}.  

\textsc{Construction of $Q_S$.}  Define $\o Q_S$ to be the pushout of the left half of the diagram of \'etale sheaves of monoids below:
\begin{equation} \label{eqn:1} \vcenter{\xymatrix{
\pi_! \pi^\ast \o M_S \ar[r] \ar[d] & \pi_! \o R \ar[r] \ar[d] & \pi_! \pi^\ast \o N_S \ar[d] &  \\
\o M_S \ar[r] \ar@/_15pt/[rr] & \o Q_S \ar@{-->}[r] & \o N_S
}} \end{equation}
The commutativity of the diagram of solid lines and the universal property of pushout induce a morphism $\o Q_S \rightarrow \o N_S$, as shown above.  Pulling back $N_S$ via this map gives a quasi-logarithmic structure $\o Q_S$ with characteristic monoid $\o Q_S$. 

\begin{remark}%
The construction of $\o Q_S$ can be simplified when $\pi  : X \rightarrow S$ has connected fibers.  Under that hypothesis, the counit $\pi_! \pi^\ast \o M_S \rightarrow \o M_S$ is an isomorphism and $\o Q_S = \pi_! \o R$.
\end{remark}
 
\textsc{Construction of $\psi$.}
Pulling diagram~\eqref{eqn:1} back to $X$ we get a commutative diagram of \'etale sheaves of monoids:
\begin{equation} \label{eqn:8} \vcenter{\xymatrix{
\pi^\ast \o M_S \ar[r] \ar[d] \ar@/_30pt/[dd]_(0.3){\id} & \o R \ar[r] \ar[d] & \pi^\ast \o N_S \ar@/^30pt/[dd]^(0.7){\id} \ar[d] \\
\pi^\ast \pi_! \pi^\ast \o M_S \ar[r] \ar[d] & \pi^\ast \pi_! \o R_S \ar[r] \ar[d] & \pi^\ast \pi_! \pi^\ast \o N_S \ar[d] \\
\pi^\ast \o M_S \ar[r] & \pi^\ast \o Q_S \ar[r] & \pi^\ast \o N_S
}} \end{equation}
The vertical arrows on the left and right sides of the diagram compose to identities because $\pi_!$ and $\pi^\ast$ are adjoint functors.  We isolate the commutative diagram 
\begin{equation*} \xymatrix{
\pi^\ast \o M_S \ar[r] \ar[dr] & \o R \ar[d] \ar[r] & \pi^\ast \o N_S \\
& \pi^\ast \o Q_S \ar[ur]
} \end{equation*}
and push it out via $\pi^\ast \o M_S \rightarrow \o M_X$ to obtain the lower half of the diagram below:
\begin{equation} \label{eqn:9} \vcenter{ \xymatrix{
& \o M \ar[d]_{\o\rho} \ar[dr]^{\o\varphi} \\
\o M_X \ar[dr] \ar[r] & \o R_X \ar[r] \ar[d] & \o N_X \\
& \o Q_X \ar[ur]
}} \end{equation}
The upper half of the diagram is provided by the morphism $(R, \rho) \rightarrow (\pi^\ast N_S,\psi)$ of $\GS^{\rm loc}(\o X)$.  By composing the vertical arrows in the center of the diagram, we obtain the definition of $\o\psi$:
\begin{equation*}
\o\psi : \o M \xrightarrow{\o\rho} \o R_X \rightarrow \o Q_X 
\end{equation*} 
Now observe that $Q_X$ is the pullback of $N_X$ via the map $\o Q_X \rightarrow \o N_X$.  The commutative triangle
\begin{equation*} \xymatrix{
\o M \ar[d]_{\o\psi} \ar[dr]^{\o\varphi} \\
\o Q_X \ar[r] & \o N_X
} \end{equation*}
yields a commutative triangle
\begin{equation*} \xymatrix{
M \ar@{-->}[d]_{\psi} \ar[dr]^{\varphi} \\
Q_X \ar[r] & N_X ,
} \end{equation*}
by the universal property of the fiber product.

\begin{proof}[Proof of Lemma~\ref{lem:minimal-existence}]
We check that the object $(Q_S,\psi)$ constructed above satisfies the universal property of a minimal object of $\GS(\u S)$.

Consider a morphism $(N'_S,\varphi') \rightarrow (N_S,\varphi)$ of $\GS(\u S)$.  We must show that there is a unique map $(Q_S,\psi) \rightarrow (N'_S,\varphi')$ rendering the triangle below commutative:
\begin{equation} \label{eqn:2} \vcenter{\xymatrix{
(Q_S,\psi) \ar@{-->}[r] \ar[dr] & (N'_S,\varphi') \ar[d] \\
 & (N_S,\varphi)
}} \end{equation}
To specify the dashed arrow above we must give a factorization of $Q_S \rightarrow N_S$ through $N'_S$ and show the induced triangle
\begin{equation} \label{eqn:3} \vcenter{\xymatrix{
& M \ar[dl]_{\psi} \ar[d]^{\varphi'} \\
Q_X \ar[r] & N'_X
}} \end{equation}
is commutative.  

The dashed arrow in diagram~\eqref{eqn:2} and the commutativity of diagram~\eqref{eqn:3} are both determined at the level of characteristic monoids.  That is, to give a factorization of $Q_S \rightarrow N_S$ through $N'_S$ is the same as to give a factorization of $\o Q_S \rightarrow \o N_S$ through $\o N'_S$ since the logarithmic structure $N'_S$ is pulled back from $N_S$ via $\o N'_S \rightarrow \o N_S$.  Similarly, to verify the commutativity of diagram~\eqref{eqn:3} it is sufficient to show that the induced diagram of characteristic monoids commutes.

By the definition of $\o Q_S$ as a pushout in diagram~\eqref{eqn:1}, to give $\o Q_S \rightarrow \o N'_S$ compatible with the tacit maps from $\o M_S$ is the same as to give $\pi_! \o R \rightarrow \o N'_S$ compatible with the maps from $\pi_! \pi^\ast \o M_S$.  The latter is equivalent, by adjunction, to giving $\o R \rightarrow \pi^\ast \o N'_S$, compatible with the maps from $\pi^\ast \o M_S$.  But by the minimality of $(R,\rho)$ in $\GS^{\rm loc}(\u X)$, there is a unique factorization of $R \rightarrow \pi^\ast N_S$ through $\pi^\ast N'_S$ such that the induced triangle depicted in the upper half of the diagram below is commutative.
\begin{equation*} \xymatrix{
M \ar[d]^{\rho} \ar[dr]^{\varphi'} \ar@/_15pt/[dd]_{\psi} \\
R_X \ar[r] \ar[d] & N'_X \\
Q_X \ar[ur]
} \end{equation*}
The lower triangle in the diagram is also commutative:  As already noted, the map $R \rightarrow \pi^\ast N'_S$ gives a factorization of $\pi_! \o R \rightarrow \o N'_S$ through $\o Q_S$, so by adjunction we get a factorization of $\o R \rightarrow \pi^\ast \o N'_S$ through $\pi^\ast \o Q_S$, and therefore a factorization $R \rightarrow \pi^\ast \o Q_S \rightarrow \o N'_S$.  The lower triangle is the pushout of this sequence via $R \rightarrow R_X$.  The outer triangle, which coincides with diagram~\eqref{eqn:3}, is therefore commutative, as desired.
\end{proof}

\begin{proof}[Proof of Lemma~\ref{lem:minimal-pullback}]
It is sufficient to treat the case where $T = S$.  Suppose, then, that $f : \u S' \rightarrow \u S$ is a morphism of schemes and set $\u X' = \u X \fpr_{\u S} \u S'$.  We show that $(f^\ast Q_S, f^\ast \psi)$ is a minimal object of $\GS(S')$.  We verify that all of the data that go into the construction of $Q_S$ are compatible with pullback:
\begin{enumerate}
\item the minimal object $(R,\rho)$ of $\GS^{\rm loc}(\u X)$ pulls back to a minimal object of $\GS^{\rm loc}(X')$ by Lemma~\ref{lem:local-pullback}; 
\item the formation of $\pi_! \o R$ is compatible with pullback by Corollary~\ref{cor:pi_!-pullback};
\item the pushout $\o Q_S$ is compatible with pullback by the preservation of colimits under pullback of sheaves;
\item the quasi-logarithmic structure $Q_S$ is formed as a fiber product of sheaves and these are preserved by pullback.
\end{enumerate}
This shows that the construction of $Q_S$ is compatible with pullback.  We make a similar verification for $\psi$:
\begin{enumerate}[resume*]
\item The compatibility of $\pi_!$ with $f^\ast$ (Corollary~\ref{cor:pi_!-pullback}) guarantees that diagram~\eqref{eqn:8} pulls back to its analogue on $X'$;
\item compatibility of pullback of sheaves with colimits guarantees the compatibility of the lower half of diagram~\eqref{eqn:9} with pullback;
\item Lemma~\ref{lem:local-pullback} and the assumption that $(R,\rho)$ be minimal in $\GS^{\rm loc}(\u X)$ guarantee that the pullback of the upper half of diagram~\eqref{eqn:9} coincides with its analogue constructed in $\GS^{\rm loc}(X')$.
\end{enumerate}
This shows that the construction of $\psi$ is compatible with pullback and completes the proof.
\end{proof}

\section{Automorphisms of minimal logarithmic structures}
\label{sec:automorphisms}

\begin{lemma}
Suppose that $(R, \rho)$ is a minimal object of $\GS^{\rm loc}(\u X)$.  Then the automorphism group of $(R,\rho)$ is trivial.
\end{lemma}
\begin{proof}
It is equivalent to show that only the identity automorphism of $R^{\rm gp}$ is compatible with both the inclusion of $\pi^\ast M_S^{\rm gp}$ and the map $\rho : M^{\rm gp} \rightarrow R_X^{\rm gp}$.  Suppose that $\alpha$ is an automorphism of $(R,\rho)$.  Then the induced morphism $\o\alpha : \o R \rightarrow \o R$ is the identity, by Lemma~\ref{lem:initial-characteristic}.

We use this to conclude that $\alpha = \id$.  As $\o\alpha = \id$, the automorphism $\alpha$ is determined by a homomorphism $\lambda : \o R^{\rm gp} \rightarrow \mathcal{O}_X^\ast$ with
\begin{equation*}
\alpha(x) = x + \log \lambda(\o x)
\end{equation*}
and $\o x$ denoting the image of $x$ under the projection $R \rightarrow \o R$.  By assumption, $\alpha$ restricts to the identity on $\pi^\ast \o M_S^{\rm gp} \subset \o R^{\rm gp}$ so $\lambda$ restricts to the trivial homomorphism on $\pi^\ast \o M_S^{\rm gp}$.  Therefore, it must factor through the quotient $\o M^{\rm gp} / \pi^\ast \o M_S^{\rm gp}$ of $\o R^{\rm gp}$ by $\pi^\ast \o M_S^{\rm gp}$.

Let $\alpha_X : R_X \rightarrow R_X$ denote the automorphism induced by pushout of $\alpha$. We investigate the condition that $\alpha_X$ commute with $\rho$ in terms of $\lambda$ and show this implies $\lambda = 1$.  Define $\lambda_X$ by the formula $\log \lambda_X = \alpha_X - \id$ and note that $\exp \lambda_X : \o R_X^{\rm gp} \rightarrow \mathcal{O}_X^\ast$ is induced by pushout from the pair of morphisms $\lambda : \o R^{\rm gp} \rightarrow \mathcal{O}_X^\ast$ and $1 : \o M_X^{\rm gp} \rightarrow \mathcal{O}_X^\ast$.  This implies $\lambda_X \circ q = 1$ and $\lambda_X \circ i = \lambda$ in the notation of the diagram below:
\begin{equation*} \xymatrix{
0 \ar[r] & \pi^\ast \o M_S^{\rm gp} \ar[r] \ar[d] & \o R_0^{\rm gp} \ar@{-->}[dl]_\epsilon \ar[r]^r \ar[d]_q & \o M^{\rm gp} \ar[r] \ar[d] \ar@{-->}[dl]_{\o\rho} & 0 \\
0 \ar[r] & \o R^{\rm gp} \ar[r]^i & \o R_X^{\rm gp} \ar[r] & \o M_{X/S}^{\rm gp} \ar[r] & 0
} \end{equation*}
That $\lambda_X \circ q = 1$ follows from the fact that $q$ factors through $\o M_X^{\rm gp} \rightarrow \o R_X^{\rm gp}$.

For any $y \in R_0^{\rm gp}$, let $r(y)$ be its image in $M^{\rm gp}$.  Then we have
\begin{align*}
0 = \alpha_X ( \rho  r(y ) ) - \rho r(y)
& = \log \lambda_X  \o\rho  \o r(\o y)  \\
& = \log \lambda_X(q(\o y) - i \circ \epsilon(\o y)) \\
& = - \log \lambda_X(i \circ \epsilon(\o y)) \\
& = - \log \lambda(\epsilon \o y) .
\end{align*}
As $\epsilon : R_0^{\rm gp} \rightarrow R^{\rm gp}$ is surjective, we conclude that $\alpha - \id = \log \lambda = 0$ so $\alpha = \id$.
\end{proof}

\begin{corollary}
The automorphism group of a minimal object of $\GS(\u S)$ is trivial.
\end{corollary}
\begin{proof}
Suppose that $\alpha$ is an automorphism of a minimal object $(Q_S,\psi)$.  Then $\pi^\ast \alpha$ is an automorphism of the object $(\pi^\ast Q_S, \psi) \in \GS^{\rm loc}(\u X)$.  Choose a map $\iota : (R,\rho) \rightarrow (\pi^\ast Q_S,\psi)$ with $(R,\rho)$ minimal.  Then by the definition of minimality, there is a commutative diagram:
\begin{equation*} \xymatrix{
(R,\rho) \ar[d]_\iota \ar@{-->}[r] & (R,\rho) \ar[d]^\iota \\
(\pi^\ast Q_S,\psi) \ar[r]^{\pi^\ast \alpha} & (\pi^\ast Q_S,\psi)
} \end{equation*}
The dashed arrow is an automorphism of $(R,\rho)$, hence must be the identity by the lemma.  Therefore, by adjunction, the diagram below must commute:
\begin{equation*} \xymatrix{
& \pi_! \o R^{\rm gp} \ar[dr] \ar[dl] \\
\o Q_S^{\rm gp} \ar[rr]^{\o\alpha} & & \o Q_S^{\rm gp}
} \end{equation*}
But, by definition, $\o Q_S^{\rm gp}$ is generated by $\pi_! \o R^{\rm gp}$ and $\o M_S^{\rm gp}$ (diagram~\eqref{eqn:1} and the subsequent discussion).  By assumption, $\alpha$ commutes with the map $\o M_S^{\rm gp} \rightarrow Q_S^{\rm gp}$ and we have just shown it commutes with the map $\pi_! \o R^{\rm gp}$.  It follows that $\o\alpha$ is the identity map.

This implies that $\alpha$ must be induced from a map $\delta : \o Q^{\rm gp}_S \rightarrow \mathcal{O}_S^\ast$, which we would like to show is trivial.  This map is induced from a map $\o Q_S^{\rm gp} / \o M_S^{\rm gp} \rightarrow \mathcal{O}_S^\ast$ since $\delta$ is trivial on $\o M_S^{\rm gp}$.  Since $\pi_!$ preserves colimits (it is a left adjoint) we use the pushout in the left half of diagram~\eqref{eqn:1} to make identifications:
\begin{equation*}
\o Q_S^{\rm gp} / \o M_S^{\rm gp} \simeq \pi_! \o R^{\rm gp} / \pi_! \pi^\ast \o M_S^{\rm gp} \simeq \pi_! (\o R^{\rm gp} / \pi^\ast \o M_S^{\rm gp}) 
\end{equation*}
By adjunction, $\delta$ is therefore induced from a map 
\begin{equation} \label{eqn:10}
\tilde{\delta} : \o R^{\rm gp} / \pi^\ast \o M_S^{\rm gp} \rightarrow \pi^\ast \mathcal{O}_S^\ast .
\end{equation}
But composing this with the map $\pi^\ast \mathcal{O}_S^\ast \rightarrow \mathcal{O}_X^\ast$ gives an automorphism of $(R,\rho)$, which must be trivial, by the lemma.  On the other hand, $\pi^\ast \mathcal{O}_S^\ast$ injects into $\mathcal{O}_X$ since $X$ is flat over $S$, so the map $\tilde{\delta}$---and by adjunction also $\delta$---must be trivial.
\end{proof}

\begin{corollary}
The functor $\Hom_{\LogSch/S}(M,M_X)$ is representable by a logarithmic algebraic space.
\end{corollary}

\begin{corollary}
The morphism $\Hom_{\LogSch/S}(X,Y) \rightarrow \Hom_{\LogSch/S}(\u X, \u Y)$ is representable by logarithmic algebraic spaces.
\end{corollary}

%
%

\appendix

\section{Integral morphisms of monoids}
\label{sec:integral}

Recall that a morphism of monoids $f : P \rightarrow Q$ (written additively) is said to be \emph{integral} if, for any $a, a' \in P$ and $b, b' \in Q$ such that
\begin{equation*}
f(a) + b = f(a') + b'
\end{equation*}
there are elements $c, c' \in P$ and $d \in Q$ such that $a + c = a' + c'$ and $b = d + f(c)$ and $b' = d + f(c')$.

Continue to assume that $f : P \rightarrow Q$ is integral and let $P \rightarrow P'$ be another morphism of monoids.  Consider the collection of pairs $(a,b)$ where $a \in P'$ and $b \in Q$, modulo the relation $(a,b) \sim (a',b')$ if there are elements $c, c' \in P$ and $d \in Q$ such that $a + c = a' + c'$ and $b = d + f(c)$ and $b' = d + f(c')$.

The proofs of the following two lemmas are omitted as they are straightforward and likely well known.

\begin{lemma}
If $f : P \rightarrow Q$ is integral then the relation defined above is an equivalence relation.
\end{lemma}

\begin{lemma} \label{lem:integral-pushout}
The monoid structure on $P' \times Q$ descends to the equivalence classes of the relation defined above and realizes the pushout of $P \rightarrow Q$ via $P \rightarrow P'$.
\end{lemma}

\section{Minimality}
\label{sec:minimality}

\subsection{Gillam's criteria}
\label{sec:criteria}

This section is included only for convenience.  All of the results here may be found in greater detail in \cite{Gillam}.  Our Proposition~\ref{prop:minimal-criteria}, below, is Descent Lemma~1.3 of op.\ cit.

Since our only application of this section is to the fibered category of logarithmic schemes, $\LogSch$, over the category of schemes, $\Sch$, we have not made any attempt to state the results below in their natural generality.  The reader who is interested in that level of generality may easily verify that all of the arguments below are valid for an arbitrary fibered category.

Let $\Sch$ denote the category of schemes and let $\LogSch$ denote the category of logarithmic schemes.  Note that $\LogSch$ is an \'etale stack (not fibered in groupoids) over $\Sch$.  Recall that a logarithmic structure on a fibered category $F$ over $\Sch$ is a cartesian section of $\LogSch$ over $F$.

Suppose that $F$ is a fibered category over $\Sch$ with a logarithmic structure $M : F \rightarrow \LogSch$.  There is an induced fibered category $\f L (F,M)$ over $\LogSch$:  The objects of $\f L (F,M)$ are pairs $(\eta, f)$ where $\eta \in F$ and $f : Y \rightarrow M(\eta)$ is a morphism in $\LogSch$ such that $\u Y = \u{M(\eta)}$ and $f$ lies above the identity morphism.

When $F$ is representable by a logarithmic scheme, this is the familiar construction that associates to $F$ the functor it represents on logarithmic schemes.

We give a characterization of the fibered categories $G$ on $\LogSch$ that are equivalent to $\f L (F,M)$ for a fibered category $F$ with logarithmic structure $M$.

An object $\xi$ of $G$ is called \emph{minimal} if every diagram of solid lines in $G(\u \xi)$
\begin{equation*} \xymatrix{
\eta \ar[dr] \ar[rr] & & \xi \\
& \omega \ar@{-->}[ur]
} \end{equation*}
lying above $\id_{\u \xi}$ admits a unique completion by a dashed arrow.

\begin{proposition} \label{prop:minimal-criteria}
A fibered category $G$ over $\LogSch$ is equivalent to $\f L(F,M)$ if and only if the following two conditions hold:
\begin{enumerate}
\item for every $\eta \in G$ there is a minimal object $\xi \in G(\u\eta)$ and a morphism $\eta \rightarrow \xi$ lying above the identity of $\u\eta$, and
\item the pullback of a minimal object of $G$ via a morphism of $\Sch$ is minimal.
\end{enumerate}
\end{proposition}
\begin{proof}
Certainly $\f L(F,M)$ has this property.  The minimal object associated to $(\alpha,f)$ is $(\alpha,\id_{\u\alpha})$.

Conversely, let $F$ be the full subcategory of minimal objects of $G$.  By assumption, this is a fibered category over $\Sch$ with a map $M : F \rightarrow \Log\Sch$ by composition of the inclusion $F \subset G$ with the projection $G \rightarrow \LogSch$.  This is cartesian over $\Sch$ because $F$ cartesian in $G$ over $\Sch$ and $G \rightarrow \LogSch$ is cartesian over $\LogSch$, hence over $\Sch$.

We have a functor $\f L(F,M) \rightarrow G$ sending $(\alpha, f)$ to $f^\ast \alpha$.  We verify this is an equivalence.  As this is a cartesian functor between fibered categories, the verification can be done fiberwise over $\Sch$.  That is, it is enough to show that the functors $\f L (F,M)(S) \rightarrow G(S)$ are equivalences for all schemes $S$.

But $G(S)$ may be identified with 
\begin{equation*}
\coprod_{\eta \in F(S)} G(S) / \eta \:\: \simeq \coprod_{\eta \in F(S)} \LogSch / M(\eta) \simeq \f L (F,M)(S) . 
\end{equation*}
\end{proof}

\subsection{Openness of minimality}
\label{sec:openness}

Unlike the previous section, this section is specific to logarithmic structures.

The proof of Proposition~\ref{prop:minimal-criteria} shows that $G$ is $\f L(F,M)$ where $F \subset G$ is the open substack of minimal objects.  The next proposition shows that when $F$ and $G$ are fibered over \emph{coherent} logarithmic schemes (in other words, when minimal objects are coherent), $F$ is an \emph{open} substack of $G$.

For any $\xi \in F$, the image of $\xi$ in $\LogSch$ is a logarithmic scheme $S$.  We refer to the logarithmic structure on $S$ also as the logarithmic structure on $\xi$.

\begin{theorem} \label{thm:openness-minimal}
Suppose that the logarithmic structure on each $\xi \in F$ is coherent.  Then $F \subset \Log(G)$ is open.
\end{theorem}
\begin{proof}
We must show that, for any $\eta \in \Log(G)$ lying above a scheme $S$, the locus in $S$ where $\eta$ is minimal is open in $S$.  Let $\xi$ be the minimal object admitting a morphism from $\eta$ and let $M_\eta \rightarrow M_\xi$ be the associated morphism of logarithmic structures.  The locus in $S$ where $\eta$ is minimal is the same as the locus where $\eta \rightarrow \xi$ restricts to an isomorphism.  Since $G$ is fibered in groupoids over $\LogSch$, the map $\eta \rightarrow \xi$ restricts to an isomorphism if and only if $M_\xi \rightarrow M_\eta$ does.  The following lemma therefore completes the proof.
\end{proof}

\begin{lemma}
Let $\alpha : M \rightarrow M'$ be a morphism of coherent logarithmic structures on a scheme $S$.  The locus in $S$ where $\alpha$ is an isomorphism is an open subset of $S$.
\end{lemma}
\begin{proof}
It is sufficient to show that a morphism that is an isomorphism at a geometric point is an isomorphism in an \'etale neighborhood of that point.  Choose charts $\o P \rightarrow \o M$ and $\o P' \rightarrow \o M'$ near a geometric point $\xi$,%
\footnote{By a chart $\o P \rightarrow \o M$ we mean a homomorphism from a constant sheaf of monoids $\o P$ to $\o M$ such that if $P$ is defined to be the extension of $\o P$ by $\mathcal{O}_S^\ast$ obtained by pulling back $M \rightarrow \o M$, the associated logarithmic structure of $P \rightarrow M \rightarrow \mathcal{O}_S$ is $M$.}
fitting into a commutative diagram:
\begin{equation*} \xymatrix{
\o P \ar[r] \ar[d] & \o M \ar[r] \ar[d] & \o M_\xi \ar[d] \\
\o P' \ar[r] & \o M' \ar[r] & \o M'_\xi
} \end{equation*}
As the monoids $\o M_\xi$ and $\o M'_\xi$ are of finite presentation (because they are finitely generated~\cite[Theorem~5.12]{RGS}), we can choose $\o P$ and $\o P'$ so that the maps $\o P \rightarrow \o M_\xi$ and $\o P' \rightarrow \o M'_\xi$ are isomorphisms, at least after shrinking the \'etale neighborhood of $\xi$.  But then $\o P \rightarrow \o P'$ is an isomorphism and $M$ and $M'$ are therefore the logarithmic structures associated to the same quasi-logarithmic structure.  In particular, they are isomorphic.
\end{proof}

\section{Explicit formulas:  by Sam Molcho} \label{app:calc}


The purpose of this appendix is to show that under certain reasonable simplifying assumptions, the minimal logarithmic structure constructed in the paper admits a rather concrete, simple description. Specifically, we will study minimal logarithmic structures in the case where we have a family of maps 
\begin{align} \label{eqn:A1} \vcenter{
\xymatrix{X \ar[r]^f \ar[d]_{\pi} & V \ar[ld] \\ S}
} \end{align}
over a geometric point $S = \Spec k$, with $\pi$ being flat and having reduced fibers, and where the logarithmic structure on $V$ is a Zariski log structure. To our knowledge, these assumptions hold in all previous work where minimal logarithmic structures have been studied, e.g the papers, Gross-Siebert \cite{GS}, Abramovich-Chen \cite{AC}, Chen \cite{Chen}, Olsson \cite{OAV}, and Ascher-Molcho \cite{AM}. In fact, in these papers the authors always begin with flat, proper families of schemes with reduced fibers, and construct a minimal logarithmic structure over each geometric point by writing down an explicit formula and then prove that a logarithmic structure is minimal over a general family if and only if it restricts to a minimal logarithmic structure over each point. In the cited papers, what are actually considered are ``absolute" families 
\begin{align*}
\xymatrix{X \ar[r]^f \ar[d]_{\pi} & V \\ S}
\end{align*}
We will show that in this absolute situation the notion of minimality defined in the present paper and the notion given in \cite{GS} (or, in fact, an evident extension of this notion) coincide.\footnote{The formula of \cite{GS} was only presented for families of nodal curves, but it works for more general $X$ and the agreement we prove here holds in that generality.}
 
\subsection{Structure of logarithmically smooth morphisms}
\label{sec:A-structure}

Consider a logarithmically smooth morphism $\pi: (X,M_X) \rightarrow (S,M_S)$ which is flat, proper, and has reduced fibers, with $S=\Spec k$ a geometric point. The geometry of $(X,M_X)$ may be rather complicated; however, it is simple to understand at the loci where the relative characteristic $\o{M}_{X/S}$ has rank $0$ or $1$. Specifically, we have 
	\begin{theorem} \label{thm:local-str}
Suppose $\pi:(X,M_X) \rightarrow (S,M_S)$ is flat, proper, and has reduced fibers, and that $S = \Spec k$ is a geometric point. Then 
\begin{enumerate}
\item Each irreducible component of $X$ is smooth near its generic point $\eta$, and $\o M_{X,\eta}=\o M_S$. \\
\item Whenever two irreducible components intersect, they intersect in a divisor of each, which we will call a node. A node then generically is the intersection of \emph{precisely two} irreducible components, and if $q$ denotes the generic point of an irreducible component of the node, we have $M_{X,q} = M_S \oplus_{\NN} \NN^2$, where $\NN \rightarrow \NN^2$ is the diagonal map and $\rho_q: \NN \rightarrow M_S$ is some homomorphism determined by $\pi$. \\
\item There are certain divisors on the smooth locus of $X$ so that at the generic point $p$ of such a divisor, we have $M_{X,p} = M_S \oplus \NN$.
\end{enumerate}
\end{theorem}
 
The divisors in (3) are the higher dimensional analogue of marked points of logarithmic curves. In other words, the structure of a logarithmic morphism on the generic points of codimension $0$ and $1$ strata is precisely the same as the structure of a logarithmic curve, as discussed by F.\ Kato. To see why these claims are true, we apply the chart criterion for logarithmic smoothness \cite{K}, to obtain \'etale locally a commutative diagram 
\begin{align*}
\xymatrix{ X \ar[r] & S \times_{\Spec{\ZZ[Q]}} \Spec{\ZZ[P]} \ar[r] \ar[d] & \Spec \ZZ[P] \ar[d] \\ & S \ar[r] & \Spec \ZZ[Q]}
\end{align*}
\noindent where the square is cartesian and the morphism from $(X,M_X)$ to the fiber product with its induced logarithmic structure is smooth on the level of underlying schemes and strict. Since smooth morphisms preserve the information of how irreducible components intersect, the problem reduces to proving these claims for the fiber of a toric morphism of toric varieties $X(F,N) \rightarrow X(\kappa,Q)$ over the torus fixed point of $X(\kappa,Q)$. Here we use the notation $X(F,N)$ for the toric variety determined by a fan $F$ in a lattice $N$, and similarily for $X(\kappa,Q)$, where we may assume $\kappa$ is a single cone. A generic component of the fiber over the fixed point of $X(\kappa)$ then corresponds to a cone $\tau \in F$ such that $\tau$ maps isomorphically to $\kappa$, and a divisor in the fiber corresponds to a cone $\sigma$ that maps onto $\kappa$ with relative dimension $1$. Condition (1) then is equivalent to saying that $\tau \cap N$ is isomorphic to $\kappa \cap Q$, since the duals of these monoids are charts for the logarithmic structures $M_{X,\eta}$ and $M_S$. Conditions $(2)$ and $(3)$ are equivalent to saying that every cone $\sigma$ that maps onto $\kappa$ with relative dimension $1$ can have either one or two faces isomorphic to $\kappa$; and, (2) if $\sigma$ has two such faces, then that $N \cap \sigma \cong (Q \cap \kappa) \times_\NN \NN^2$ for some homomorphism $e_q: Q \cap \kappa \rightarrow \NN$, while if $(3)$ $\sigma$ has precisely one such face, then $\sigma \cap N \cong (Q \cap \kappa)\times \NN$. These statements, and the construction of the homomorphism $e_q$ are precisely the content of lemmas 3.11 and 3.12 of Ascher-Molcho \cite{AM}, or equivalently lemmas 8 and 9 in Gillam-Molcho \cite{GM2} from which lemmas 3.11 and 3.12 are derived. \\

The description given above holds \'etale locally on $X$, for the \'etale sheaf $M_X$. In what follows, we also need to understand the induced Zariski sheaf $\tau_*M_X$ on $X$, where $\tau$ denotes the morphism of sites $\et(X) \rightarrow \zar(X)$---c.f.\ lemma 4.6 and proposition 4.7. We claim that the description of $\tau_*M_X$ over the generic points $\eta$ of irreducible components of $X$ and over generic points $q$ of irreducible components of nodes is only slightly more involved. Specifically, we have 
\begin{corollary}  \label{cor:pushforward}
The logarithmic structure $\tau_*M_X$ satisfies $\tau_*M_{X,\eta} = M_S$ on the generic point of an irreducible component $\eta$. On the generic point of an irreducible component of a node, we have either
\begin{enumerate}
\item $\tau_*M_{X,q} = M_S \oplus_\NN \NN^2$, when the node is the intersection of two distinct irreducible components of $X$ in the Zariski topology, or \\
\item $\tau_*M_{X,q}=M_S$, when the node is the self-intersection of a single irreducible component.
\end{enumerate} 
\end{corollary}
\begin{proof}
	We prove the statement about nodes first. The question is local on $X$, so we may assume for simplicity that $X$ is the spectrum of $\mathcal{O}_{X,q}$ According to theorem \ref{thm:local-str} above, we can choose an \'etale cover $U$ of $X$, and a lift $q'$ of $q$, with the property that $M_{U,q'} \cong M_S \oplus_\NN \NN^2$; here the two generators $e_1,e_2$ of $\NN^2$ correspond to the two branches of $U$ around $q'$, and map to the two functions in $\mathcal{O}_{U,q}$ which define these two branches. For every \'etale cover $V$ of $U$ over $X$, and cover $q''$ of $q'$, we have $M_{V,q''} \cong M_S \oplus_\NN \NN^2$ as well, as $M_{U,q'}$ and $M_{V,q''}$ are both pulled back from $M_{X,q}$. Thus, every \'etale cover $V$ of $U$ over $X$ induces an automorphism of $M_S \oplus_\NN \NN^2$. On the other hand, the sheaf $\tau_*M_X$ is determined from $M_U$ by descent, hence $\tau_*M_{X,q}$ is the submonoid of invariants of $M_S \oplus_\NN \NN^2$ under all automorphisms of $M_S \oplus_\NN \NN^2$ obtained by \'etale covers $V \rightarrow U$ over $X$. Every such automorphism further lives over $S$, hence must fix $M_S$; and since $e_1+e_2 \in \NN^2$ is in $M_S$, such an automorphism must fix $e_1+e_2$. We are thus looking at automorphisms of $\NN^2$ which fix $(1,1)$, that is, matrices with coefficients in $\NN$, determinant $1$, and that fix $(1,1)$. The only two such matrices are the identity and the matrix $e_1 \rightarrow e_2, e_2 \rightarrow e_1$. Thus, the invariants of $M_S \oplus_\NN \NN^2$ are either all of $M_S \oplus_\NN \NN^2$ or the diagonal $M_S \oplus_\NN \NN(e_1+e_2) \cong M_S$. To prove the corollary it remains to analyze when each case happens. Suppose first that $q$ is the intersection of two distinct irreducible components of $X$. Since $\mathcal{O}_X$ is determined from $\mathcal{O}_U$ by descent as well, we see that the images of $e_1,e_2$ in $U$ map to distinct functions in $\mathcal O_X$, which define $q$ in each of the irreducible components. Thus, there can be no \'etale cover $V$ of $U$ over $X$ which interchanges $e_1,e_2$ in the automorphism $M_S \oplus_\NN \NN^2$, and we are in the situation where the invariants are all of $M_S \oplus_\NN \NN^2$. On the other hand, if $q$ is the self intersection of a single component, the only linear combinations of $e_1$ and $e_2$ that descend are the $k e_1 + k e_2$, which lie in $M_S$.

To see the statement about the logarithmic structure over the generic points of irreducible components $\eta$, we simply observe that the invariants of $M_S$ over $M_S$ are always $M_S$, hence $\tau_*M_{X,\eta} = M_S$ as well.  
\end{proof}

\subsection{The Minimal Log Structure Over a Geometric Point}

We consider a family of logarithmic schemes 
\begin{align} \vcenter{
\xymatrix{X \ar[r]^f \ar[d]_{\pi} & V \ar[ld] \\ S}
} \end{align}
\noindent with $S$ a geometric point, and $\pi$ a logarithmic smooth morphism. The morphism $f$ is not assumed to be a logarithmic morphism. We write $M_S$ for the logarithmic structure on $S$, $M_X$ for the logarithmic structure on $X$, and $M_V$ for the logarithmic structure on $V$. We further write $M=f^*M_V$ for the pullback of the logarithmic structure on $V$ to $X$. 
	
	Out of $X$, we can create a category $\mathcal{C}$.  The objects of $\mathcal C$ are the generic points of the strata in the minimal stratification on which the relative characteristic $M_{X}$ is locally constant.  Note that $\tau_\ast M_{X/S}$ is constant on these strata.  We have a morphism $x \rightarrow y$ in $\mathcal{C}$ whenever $\{x\} \in \bar{\{y\}}$.  

In the special case when $X$ is a nodal curve, the category $\mathcal{C}$ is rather simple, with one object for each irreducible component $\eta$ of $X$ and an object for each node or marked point $q$, and a morphism $q \rightarrow \eta$ whenever $q$ belongs to the component $\eta$. In fact, the same characterization holds for general $X$ in depths $0$ and $1$: depth $0$ objects correspond to irreducible components of $X$, and depth $1$ objects correspond to either marked divisors in $X$ or nodes where irreducible components of $X$ intersect, according to theorem \ref{thm:local-str}.

\begin{lemma} \label{lem:colimit-pushforward}
	Let $F$ be a sheaf in the \'etale topology on $X$ that is pulled back from a sheaf in the Zariski topology, which is constructible with respect to the category $\mathcal{C}$. Then 
\begin{align*}
\pi_{!}(F) = \varinjlim_{x \in \mathcal{C}}F_x
\end{align*}
\end{lemma}

\begin{proof}
	By Proposition~\ref{prop:zar-push}, $\pi_{!}$ is the left adjoint to $\pi^{*}_{\zar}$ the pullback functor \emph{on Zariski sheaves}, so we only need to verify that for any sheaf $G$ on $S$, $\Hom (\varinjlim_{x \in \mathcal{C}}F_x, G) = \Hom (F, \pi^{*}G)$. A sheaf $G$ on $S$ is simply a monoid, so $\pi^{*}G$ is the constant sheaf on $X$ associated to $G$. Thus, to give a homomorphism $F \rightarrow \pi^{*}G$ is equivalent to giving homomorphims $F_x \rightarrow G$ for each $x \in \mathcal{C}$ which are compatible with generization; but this is precisely the data of a homomorphism $\varinjlim_{x \in \mathcal{C}} F_x \rightarrow G$.  
\end{proof}

Observe furthemore that the colimit over a finite indexing category only depends on the objects of depths $0$ and $1$, i.e in this case over the irreducible components $\eta$ of $X$, the generic points $q$ of the nodes of $X$, and the generic points of the marked divisors of $X$.  Note that the marked divisors do not contribute to the colimit.  The reason is that for each $p \in \mathcal{C}$ corresponding to a marked divisor, there is a unique morphism $p \rightarrow \eta$, where $\eta$ corresponds to the irreducible component containing the marked divisor. So the points corresponding to marked divisors can be ommited from the diagram without affecting the colimit. The same is true for nodes of $X$ which are the self intersection of a single irreducible component. 

From here on, $\eta$ is always going to denote the generic point of an irreducible component of $X$, and $q$ is always going to denote the generic point of an irreducible component of the intersection of two \emph{distinct} irreducible components.

In what follows, we will replace $M_X$ with the sheaf $\tau^*\tau_*M_X$. Note that while $X$ equipped with $\tau^*\tau_*M_X$ is no longer log smooth over $(S,M_S)$, the minimal log structure obtained from $\tau^*\tau_*M_X$ and the minimal log structure obtained from $M_X$ coincide, according to \ref{prop:zar-GS}. The reason we do this replacement is because this allows us to use $\mathcal{C}$ in the calculation of the minimal log structure, according to lemma \ref{lem:colimit-pushforward}, rather than the far larger category of all \'etale specializations. Furthermore, according to the preceeding remark, only irreducible components $\eta$ and nodes $q$ need to be included in the calculation and there we have $\tau^*\tau_* M_{X,\eta} = M_{X,\eta}, \tau^*\tau_*M_{X,q} = M_{X,q}$ according to corollary \ref{cor:pushforward}. So this is in fact not a serious abuse of notation. 

This allows us to work an explicit presentation for the minimal logarithmic structure. We first determine its associated group. Let us recall the notation. Diagram~\eqref{eqn:A1} gives us a diagram 
\begin{align*}
\xymatrix{
M_{X/S}^{\rm gp} & \o M_X^{\rm gp} \ar[l] &  \ar[l] \ar@/_20pt/[ll]_u M^{\rm gp} \\
& \pi^*M_S^{\rm gp} \ar[u] \ar[ur]_{i}}
\end{align*} 

\noindent where $M = f^{*}M_V$ and $u$ is a given homomorphism which is fixed in advance, part of the discrete geometric data of the problem -- the type of the morphism. The problem of finding a minimal logarithmic structure of the data is the same as finding a minimal object $N_S$ such that $u$ factors through a homomorphism $M^{\rm gp} \rightarrow N_S^{\rm gp} \oplus_{\pi^*M_S^{\rm gp}} M_X^{\rm gp}$. To find the minimal logarithmic structure, we set

\begin{align*}
\o R_{0}^{\rm gp} = \o M_X^{\rm gp} \times_{\o M_{X/S}^{\rm gp}} \o M^{\rm gp}
\end{align*}

\noindent and 

\begin{align*}
\o R^{\rm gp} = \o R_0^{\rm gp} / \Delta (\pi^{*}\o M_S^{\rm gp})
\end{align*}
where $\Delta$ is the diagonal map $\pi^\ast \o M_S^{\rm gp} \rightarrow \o M_X^{\rm gp} \times_{\o M_{X/S}^{\rm gp}} \o M^{\rm gp}$.

Applying the lemma, we see that the associated  group of the minimal logarithmic structure $\pi_! \o R^{\rm gp}$ on $S$ is given as the coequalizer

\begin{align*}
\varinjlim_{x \in C} \o R_x^{\rm gp} = \varinjlim \Bigl( \xymatrix{\displaystyle \bigoplus_{q} \o R^{\rm gp}_q \ar[r]<2pt>^<>(0.5){\phi_{1}} \ar[r]<-2pt>_<>(0.5){\phi_2} & \displaystyle \bigoplus_{\eta} \o R_\eta^{\rm gp} \Bigr) } 
\end{align*}  

\noindent where $\phi_1,\phi_2$ denote the two generization maps, induced by the generization morphisms $M_{X,q} \rightarrow M_{X,\eta}$ and $M_q \rightarrow M_\eta$ whenever $q$ is contained in $\eta$. If we choose an ordering of the two components $\eta_1,\eta_2$ containing a node $q$, the coequalizer becomes the quotient 

\begin{align*}
\xymatrixcolsep{5pc}
\bigoplus_{q} \o R_q^{\rm gp} \xrightarrow{\phi_2-\phi_1} \bigoplus_{\eta} \o R_\eta^{\rm gp} 
\end{align*} 

On the other hand, even though $R$ may be difficult to understand, its stalks at generic points and nodes are straightforward. We have 
\begin{align*}
\o R_{0,\eta}^{\rm gp} = \o M^{\rm gp}_{X,\eta} \times \o M^{\rm gp}_\eta \cong \o M^{\rm gp}_S \times \o M_{\eta}^{\rm gp} 
\end{align*}
and so 
\begin{align*}
\o R_{\eta}^{\rm gp} =\o M_{\eta}^{\rm gp}
\end{align*}
Similarily, on a node $q$ we have 
\begin{align*}
\o R_{0,q}^{\rm gp} = \o M_{X,q}^{\rm gp} \times_\ZZ \o M_{q}^{\rm gp}
\end{align*}
Now recall that $M_{X,q} = M_{S} \oplus_\NN \NN^2 = \{(a,b):b-a=\rho_qd\} \subset M_S \times M_S$ where $\rho_q:\NN \rightarrow M_S$ is a homomorphism determined by $\pi$, as discussed in section~\ref{sec:A-structure}. We abusively also write $\rho_q$ for the image $\rho_q(1) \in M_S$ of $\rho_q$. The morphism $M_{X,q}^{\rm gp} \rightarrow \ZZ$ is the canonical projection $M_{X,q}^{\rm gp} \rightarrow M_{X/S}^{\rm gp}$, which, explicitly, is the homomorphism that sends a pair $(a,b)$ such that $b-a = \rho_q d$ to $d$.  Therefore, 
\begin{align*}
\o R_{0,q}^{\rm gp} = \bigl\{(a,b,m): b-a =u_q(m)\rho_q\bigr\}
\end{align*}
The morphism $\pi^{*}M_S^{\rm gp} \rightarrow R_{0,q}^{\rm gp}$ is the diagonal $s \mapsto (s,s,i(s))$, and hence the quotient 
\begin{align*}
\o R_q^{\rm gp} \cong \o M_q^{\rm gp}
\end{align*}
\noindent the isomorphism sending $[(a,b,m)] \mapsto m-i(a)$, with inverse $m \mapsto [0,u_q(m)\rho_q,m]$. The two generization morphisms $\o R_{0,q}^{\rm gp} \rightarrow \o R_{0,\eta}^{\rm gp} \rightarrow \o R_{\eta}^{\rm gp}$ are the two natural maps from $\o M_{X,q}^{\rm gp} \times_\ZZ \o M_{q}^{\rm gp}$ to $\o M_\eta^{\rm gp}$, sending $(a,b,m)$ to $(i(a)+\phi_1(m))$ or to $(i(b)+\phi_2(m))$ respectively, depending on whether $\eta$ is the first or second irreducible component containing $q$, under our ordering. Thus, under the isomorphism above, the morphism $\o R_q^{\rm gp} \rightarrow \o R_{\eta_1}^{\rm gp} \times \o R_{\eta_2}^{\rm gp}$ becomes the map $\o M_q^{\rm gp} \rightarrow \o M_{\eta_1}^{\rm gp} \times \o M_{\eta_2}^{\rm gp}$ which sends $m$ to $\bigl(-\phi_1(m),\phi_2(m)+i(u_q(m)\rho_q)\bigr)$. The associated group of the minimal logarithmic structure thus has the very simple quotient presentation: 

\begin{align*}
\bigoplus_{q} \o M_q^{\rm gp} \xrightarrow{(-\phi_1,\phi_2,iu_q\rho_q)} \bigoplus_{\eta} \o M_\eta^{\rm gp} 
\end{align*} 
\noindent 

\noindent The characteristic monoid of the actual logarithmic structure $\pi_! \o R$ is then the image of $\oplus \o M_{\eta}$ in $\pi_! \o R^{\rm gp}$.

\subsection{Gross-Siebert Minimality}
We now specialize to the case when the family of morphisms $f:X \rightarrow V$ is absolute, i.e of the form 
\begin{align} \label{eqn:A2} \vcenter{
\xymatrix{X \ar[r]^f \ar[d]_{\pi} & V \\ S}
} \end{align}
\noindent This point of view can be reconciled with that of the previous paragraph by simply redefining the logarithmic structure of $V$ to be $M'_V = M_V \oplus \pi^{*}M_S$. At any rate, to avoid confusion and keep in line with existing literature, we will denote the sheaf of monoids $f^{*}\o M_V$ by $P$, and $f^{*}M'_V$ by $M$, as above. For a node $q$, specializing two irreducible components $
\eta_1,\eta_2$, let $\chi_i:P_q \rightarrow P_{\eta_i}$ denote the two generization maps.

\begin{definition}
Let $S$ be a geometric point. A logarithmic structure $N_S$ on $S$ is called a \emph{Gross-Siebert minimal logarithmic structure} for $f$ over $M_S$ if its characteristic monoid $Q = \o N_S$ has associated group isomorphic to the cokernel of
\begin{align*}
\Phi : \bigoplus_{q} P_q^{\rm gp} \xrightarrow{(-\chi_1,\chi_2,\rho_qu_q)} \o M_S^{\rm gp} \oplus \bigoplus_{\eta} P_\eta^{\rm gp}   
\end{align*} 
and $Q$ is isomorphic to the image of $\o M_S^{\rm gp} \oplus \bigoplus_{\eta} P_{\eta}^{\rm gp}$ in the associated group. 
\end{definition}

\begin{remark}
In \cite{GS}, the minimal logarithmic structure is in fact saturated by definition, i.e $Q$ is the saturation of the monoid in the definition above. Since our results go through in case $Q$ is merely integral, the definition is stated in this slightly more general form. 
\end{remark}

\begin{remark}
In \cite{GS}, the object of study is stable logarithmic maps, in which case $X \rightarrow S$ is a logarithmic curve. In this case, there are canonical logarithmic structures on $X$ and $S$ which make $X \rightarrow S$ logarithmically smooth. Over a geometric point, the characteristic monoid of the canonical logarithmic structure on $S$ is simply $\NN^{m}$, where $m$ is the number of nodes of $X$. The definition of minimality given in \cite{GS} has this canonical logarithmic structure as $M_S$ throughout. The definition presented here is the evident modification that allows the same flexibility as the present paper.  
\end{remark}

\begin{theorem}
A Gross-Siebert minimal logarithmic structure is minimal. 
\end{theorem}

\begin{proof}
Let $(X,N_X) \rightarrow (S,N_S)$ be any log morphism pulled back from $(X,M_X) \rightarrow (S,M_S)$ which admits a log morphism $f$ to $V$, as in diagram~\eqref{eqn:A2}.  The morphism $f$ induces morphisms $P_{\eta}^{\rm gp} \rightarrow \o N_{X,\eta}^{\rm gp} = \o N_S^{\rm gp}$ for each $\eta$. We thus have a unique extension to a summation morphism $\Sigma: \o M_S^{\rm gp} \oplus \bigoplus_{\eta} P_{\eta}^{\rm gp} \rightarrow \o N_S^{\rm gp}$. On the other hand, for each node $q$, the diagram
\begin{align*}
\xymatrix{& P_{\eta_1}^{\rm gp} \ar[r] & \o N_S^{\rm gp} \ar@/^15pt/[dr] &  \\ P_q^{\rm gp} \ar[ur] \ar[dr] \ar@{-->}[rrr] & & & \o N_{X,q}^{\rm gp}=\{(a,b):b-a=\rho_qd\} \\ & P_{\eta_2}^{\rm gp} \ar[r] & \o N_S^{\rm gp} \ar@/_15pt/[ur] &} 
\end{align*}
must commute; thus the map $\Sigma$ must descend uniquely to the (unsharpened) quotient of $\Phi$. As $N_{S}^{\rm gp}$ is assumed torsion free, the morphism must descend to the sharpened quotient as well, i.e, the associated group of the minimal logarithmic structure, as desired. Since $\o M_S \oplus \bigoplus_{\eta} P_{\eta} \rightarrow \o N_{S}^{\rm gp}$ factors through $\o{N}_S$, its image $Q$ in $Q^{\rm gp}$ maps to $\o N_S$, as desired. 
\end{proof}

We are now ready for the comparison. 

\begin{theorem}
Suppose $(X,M_X) \rightarrow (S=\Spec \CC,M_S)$ is a logarithmic smooth morphism which is flat and has reduced fibers, and $f: X \rightarrow (V,M_V)$ is a morphism to a logarithmic scheme. The minimal logarithmic structure defined in the paper coincides with the Gross-Siebert minimal logarithmic structure. 
\end{theorem}

\begin{proof}
This is immediate as both logarithmic structures satisfy the same universal property. 
\end{proof}

It is actually rather interesting to give a direct proof of this fact, as the calculation is illuminating. We consider

\begin{align*}
\xymatrix{X \ar[r]^{f} \ar[d]_{\pi} & V \\ S & }
\end{align*}

\noindent and extend it to

\begin{align*}
\xymatrix{X \ar[r]^{f} \ar[d]_{\pi} & V \ar[dl] \\ S & }
\end{align*}

\noindent as above by putting $M_V'=M_V \oplus \pi^*M_S$, $M=f^*M'_V$.\marg{should just write $V \times S$ in the second diagram; otherwise need to emphasize that $\u S$ is the final scheme} Then, $M_q = P_q \oplus M_S$ and $M_\eta = P_\eta \oplus M_S$. The morphism $m \mapsto (-\phi_1(m),\phi_2(m)+i(u_q(m)\rho_q))$ is identified with $(m,s) \mapsto (-\chi_1(m),-s,\chi_2(m),u_q(m)\rho_q + s)$. Therefore, by the results of section $1$, we have that the associated group of the minimal logarithmic structure for the family is given as the quotient

\begin{align*}
\xymatrixcolsep{10pc}\xymatrix{\displaystyle \bigoplus_{q} \bigl( P_q^{\rm gp} \oplus M_S^{\rm gp} \bigr) \ar[r]^<>(0.5){(-\chi_1,-id,\chi_2,id+u_q\rho_q)} & \displaystyle \bigoplus_{\eta} \bigl( P_\eta^{\rm gp}} \oplus M_S^{\rm gp}  \bigr)
\end{align*} 

\noindent Then, if $\Sigma$ denotes the summation map $\oplus M_S^{\rm gp} \rightarrow M_S^{\rm gp}$, we obtain a commutative diagram

\begin{align*}
\xymatrixcolsep{10pc}\xymatrix{\raisebox{0pt}[0pt][5pt]{$\displaystyle \bigoplus_{q} \bigl( P_q^{\rm gp} \oplus \o M_S^{\rm gp} \bigr)$} \ar[d] \ar[r]^<>(0.5){(-\chi_1,-id,\chi_2,id+u_q\rho_q)} & \raisebox{0pt}[0pt][5pt]{$\displaystyle \bigoplus_{\eta} \bigl( \o M_S^{\rm gp} \oplus P_{\eta} \bigr)$} \ar[d]^{(\Sigma,\id)} \\
\displaystyle \bigoplus_q P^{\rm gp}_q \ar[r]_{(u_q\rho_q,-\chi_1,\chi_2)} &  \displaystyle \o M_{S}^{\rm gp} \oplus \bigoplus_{\eta} P_{\eta}}
\end{align*}

\noindent Thus, the morphism $(\Sigma, \id)$ descends to a morphism of the quotients $\pi_! \o R^{\rm gp} \rightarrow Q_S^{\rm gp}$, where $Q$ is the characteristic monoid of the Gross-Siebert minimal logarithmic structure: 

\begin{align*}
\xymatrixcolsep{5pc}\xymatrix{\raisebox{0pt}[0pt][5pt]{$\displaystyle \bigoplus_{q} P_q^{\rm gp} \oplus \o M_S^{\rm gp}$} \ar[d] \ar[r]  & \raisebox{0pt}[0pt][5pt]{$\displaystyle \bigoplus_{\eta} \bigl( \o M_S^{\rm gp} \oplus P_{\eta} \bigr)$} \ar[d]^{(\Sigma,\id)} \ar[r]& \pi_! \o R^{\rm gp} \ar[d] \\
\displaystyle \bigoplus_q P^{\rm gp}_q \ar[r] & \displaystyle \o M_{S}^{\rm gp} \oplus \bigoplus_{\eta} P_{\eta} \ar[r] & Q_S^{\rm gp}}
\end{align*}

Observe that the kernel of the map $P_{q}^{\rm gp} \oplus M_S^{\rm gp} \rightarrow P_{q}^{\rm gp}$ is simply $M_S^{\rm gp}$. The kernel of the map on the right on the other hand is the kernel of the summation map $\oplus_{\eta} M_{S}^{\rm gp} \rightarrow M_S^{\rm gp}$. The induced map of kernels then fits into the sequence 

\begin{align*}
\bigoplus_{q} M_S^{\rm gp} \rightarrow \bigoplus_{\eta} M_S^{\rm gp} \rightarrow M_{S}^{\rm gp}
\end{align*} 

But this is precisely the part of the complex that computes the homology of the geometric realization $B\mathcal{C}$ of the category $\mathcal{C}$ with coefficients in the group $\o M_S^{\rm gp}$, where the morphism on the right is the evaluation map. Since $X$ is connected, so is $\mathcal{C}$, and thus $H_0(B\mathcal{C},\o M_S^{\rm gp})=\o M_S^{\rm gp}$, and thus $\oplus_q \o M_S^{\rm gp}$ surjects onto the kernel of the evaluation map $\Sigma$. Thus, the map $\pi_!\o R^{\rm gp} \rightarrow Q$ is injective. On the other hand, the morphism $(\Sigma,\id)$ is certainly surjective, so $\pi_!\o R^{\rm gp} \rightarrow Q^{\rm gp}$ is also surjective. We thus get the isomorphism on the level of associated groups. \\

To extend the isomorphism on the level of actual monoids, note that since the summation morphism is also surjective on the level of monoids, the image of $\bigoplus_{\eta} (M_S^{\rm gp} \oplus P_\eta)$ surjects onto the image of $\o M_S \oplus(\bigoplus_\eta P_\eta)$, and hence $\pi_! \o R$ surjects onto $Q$. Since $\pi_! \o R$ and $Q$ are integral and the morphism of associated groups is injective, we obtain the desired isomorphism.  


\bibliographystyle{amsalpha}
\bibliography{minimal}

\end{document}